\documentclass[11pt,leqno]{amsart}
\usepackage{amssymb,amsmath,amsthm,dsfont,fourier}

\usepackage{mathtools}
\usepackage{array}
\usepackage[text={6.5in,8.5in},centering]{geometry}
\usepackage[dvipsnames]{xcolor}
\usepackage{enumitem}
\usepackage{tikz-cd}
\usetikzlibrary{decorations.pathmorphing}
\usepackage[colorlinks=true,citecolor=NavyBlue,linkcolor=NavyBlue,urlcolor=NavyBlue]{hyperref}

\theoremstyle{plain}
\newtheorem{theorem}{Theorem}[section]
\newtheorem{maintheorem}{Theorem}

\newtheorem{proposition}[theorem]{Proposition}
\newtheorem{lemma}[theorem]{Lemma}
\newtheorem{corollary}[theorem]{Corollary}
\theoremstyle{definition}
\newtheorem{definition}{Definition}[section]
\newtheorem{example}[theorem]{Example}
\newtheorem{hypothesis}[theorem]{Hypothesis}

\theoremstyle{remark}
\newtheorem{remark}[theorem]{Remark}
\numberwithin{equation}{section}

\DeclareMathOperator{\Id}{Id}

\DeclareMathOperator{\Max}{Max}
\DeclareMathOperator{\Maxk}{Max_k}
\DeclareMathOperator{\Spec}{Spec}
\DeclareMathOperator{\pt}{pt}
\DeclareMathOperator{\Cong}{Cong}
\DeclareMathOperator{\ZCong}{ZCong}
\DeclareMathOperator{\MaxC}{MaxCong}
\newcommand{\ZId}{\mathsf{ZId}}
\newcommand{\GCl}{\mathsf{Id}_{g}}
\newcommand{\KCSpec}{\mathsf{KCSpec}}
\newcommand{\GCSpec}{\mathsf{CSpec}_{g}}
\newcommand{\M}{\mathcal M}
\newcommand{\Mk}{\mathcal M_k}
\newcommand{\MC}{\mathcal M_{\mathcal C}}
\newcommand{\mfrak}{\mathfrak m}
\newcommand{\clz}{\mathfrak{cl}_{z}}
\newcommand{\clc}{\mathfrak{cl}^{c}_{z}}
\newcommand{\ideal}[1]{\langle #1\rangle}
\newcommand{\zero}[1]{0_{#1}}
\newcommand{\wpI}[1]{\wp_{#1}}

\newcommand{\two}{\mathbf 2}

\newcommand{\N}{\mathds N}
\newcommand{\doi}[1]{\href{https://doi.org/#1}{\nolinkurl{doi:#1}}}

\begin{document}

\title[Spectral theory of $z$-ideals in semirings]%
{Maximal-Hull $z$-Ideals, Congruence Closures,\\
and Coherent Frames of Commutative Semirings}

\author{Pubali Sengupta}
\address{Department of Mathematics, Jadavpur University,
188 Raja Subodh Chandra Mallick Road, Jadavpur,
Kolkata, West Bengal 700032, India}
\email{pubalis.math.rs@jadavpuruniversity.in}

\author{Amartya Goswami}
\address{Department of Mathematics and Applied Mathematics,
University of Johannesburg,
P.O.\ Box 524, Auckland Park 2006, South Africa;
National Institute for Theoretical and Computational Sciences
(NITheCS), South Africa}
\email{agoswami@uj.ac.za}

\author{Pronay Biswas}
\address{Department of Mathematics, Jadavpur University,
188 Raja Subodh Chandra Mallick Road, Jadavpur,
Kolkata, West Bengal 700032, India}
\email{pronayb.math.rs@jadavpuruniversity.in, pronaybiswas1729@gmail.com}

\author{Sujit Kumar Sardar$^\dagger$}
\thanks{$^\dagger$ Corresponding author: Sujit Kumar Sardar, \href{mailto:sujitk.sardar@jadavpuruniversity.in}{sujitk.sardar@jadavpuruniversity.in}.}
\address{Department of Mathematics, Jadavpur University,
188 Raja Subodh Chandra Mallick Road, Jadavpur,
Kolkata, West Bengal 700032, India}
\email{sujitk.sardar@jadavpuruniversity.in, sksardarjumath@gmail.com}

\subjclass[2020]{16Y60, 06D22, 13A15, 54F65}
\keywords{$z$-ideal, $z$-congruence, $g$-closed ideal,
  spectral space, coherent frame, semiring}

\begin{abstract}
We develop a spectral theory of $z$-ideals for commutative semirings.
The lattice $\ZId(S)$ of $z$-ideals is a \emph{coherent frame} for
every commutative semiring $S$---unconditionally, without
cancellativity, subtractivity, or Noetherian hypothesis---so the
prime spectrum $\Spec_z(S)$ is spectral.
Under an explicit finite-type hypothesis on the canonical
congruence-generated closure~$g$, the lattice $\GCl(S)$ of
$g$-closed ideals is likewise a coherent frame, and $\Spec_g(S)$ is
spectral and homeomorphic to the space of prime $g$-congruences.
These frame results are accompanied by a regularity criterion:
a semiring with all multiplicative idempotents complemented is von
Neumann regular if and only if every principal ideal is a $z$-ideal,
extending Mason's classical theorem from rings.
Separating the maximal-ideal-hull $z$-closure from the
maximal-congruence-hull $g$-closure---operations that coincide in
rings but diverge in semirings---is a central theme, confirmed by
explicit computations in $\N$ and power-set semirings.
Both constructions carry a complete functorial formulation.
\end{abstract}

\maketitle

\setcounter{tocdepth}{1}
\tableofcontents

\section{Introduction}

The theory of $z$-ideals is one of the places where algebra,
topology, and order theory converge most naturally.
In a ring of continuous functions, the guiding geometric intuition
is as follows: an ideal should contain a function whenever the
zero-set information already encoded in the ideal forces that
function's zero-set to be contained in every zero-set of the ideal's
members.
This viewpoint is classical in the theory of $C(X)$ and its
Stone--\v{C}ech compactification~\cite[Chapter~2]{GillmanBook}.
For general commutative rings with identity, Mason gave the maximal-hull
reformulation~\cite{Mason71/z-ideals/prime-ideals,Mason2}:
an ideal is a $z$-ideal if it is closed under replacement of any
of its elements by any other element lying in exactly the same set of
maximal ideals.
In this language, Mason's regularity theorem states that a commutative
ring is von Neumann regular precisely when every ideal is a $z$-ideal.
The present paper asks how far this maximal-hull philosophy can be
transported from rings to semirings without invoking additive inverses.

That last qualification is essential.
A commutative semiring has ideals, prime ideals, maximal ideals,
quotients, and localizations, but the absence of additive inverses
alters the formal behaviour of each of these objects in subtle ways.
Subtractive ideals ($k$-ideals) and $k$-congruences are not merely
technical decorations: they are the mechanisms for recovering
ring-like exactness in a setting where an ideal and the congruence
it generates no longer determine one another.
This departure from ring theory is not a deficiency but a feature,
since semirings arise naturally in tropical geometry,
formal language theory, optimization, and in spectra of rings
under partial quotient constructions---all settings where additive
inverses are genuinely absent.
General background on semiring algebra may be found in the
monographs~\cite{Hebisch,Golan/Book}; the ideal-theoretic and
spectral side has a well-established lineage:
scheme-theoretic spectra provided one of the original sources of
the modern language~\cite{Grothendieck}, prime spectra of rings
motivated the abstract notion of spectral spaces~\cite{Hochster},
and coherent frames and nuclei provide the pointfree formulation
used throughout this paper~\cite{Johnstone,Banaschewski88,spectral-book}.
Constructive treatments of commutative algebra are also relevant,
since they make the hypotheses of finite generation, radical operations,
and lattice-theoretic exactness explicit at precisely the points
where our semiring arguments require them~\cite{Lombardi/Book}.

The semiring literature on $z$-ideals is comparatively recent.
Classes of ideals in semirings and their applications were introduced
in~\cite{BiswasFilomat}; $z$-ideals and $z$-closure operations
for semirings were developed in~\cite{AG23}; the exchange theorem
identifying $z$-prime ideals with prime $z$-ideals was
proved in~\cite{AG25}.
On a parallel track, $z$-congruences for the positive cone $C^+(X)$
were studied in~\cite{BiswasPositivity};
functional representation methods for semirings provide an additional
geometric source~\cite{Chermnykh12};
and prime-congruence techniques appear in dimension theory for
polynomial semirings~\cite{Joo/Micheva/dimension}.
Lattice- and frame-theoretic aspects of $z$-ideals were studied
for $f$-rings in~\cite{DubeZ-ideals,DubeRadical},
for commutative rings in~\cite{Ighedo},
and for multiplicative lattices in~\cite{GoswamiDube}.
The point addressed here is the formulation of a coherent-frame
and spectral-space theory that keeps the ideal-theoretic and
congruence-theoretic closures \emph{separate} when they fail to coincide---which they do, already in~$\N$.

Three related difficulties guide the results of this paper.
First, the absence of subtraction means that the ordinary $z$-closure
and the congruence-generated $g$-closure, which agree in rings,
may differ substantially in semirings:
a maximal congruence can have a zero-class that is neither a maximal
ideal nor a $z$-ideal.
This divergence forces a careful separation between three closure
operations---the maximal-ideal-hull $z$-closure,
the maximal-$k$-ideal-hull $z_k$-closure,
and the maximal-congruence-hull $g$-closure---which coincide in
rings but separate already in $\N$.
Second, coherence of the $z$-ideal lattice, which in the ring case
follows from Noetherian-type conditions or from the algebra of
$C(X)$, must here be established from first principles in full
generality.
Third, the congruence-generated closure satisfies the requisite
finite-type and product-formula properties only under additional
structural hypotheses that we make fully explicit.

The first main result is a regularity criterion.
In rings, the Boolean behaviour of idempotents in von Neumann regular
contexts is automatic.
For semirings, the semiring analogue of von Neumann
regularity~\cite{Subramanian70/RegularSemirings} requires the explicit
hypothesis that every multiplicative idempotent is complemented.
Under this hypothesis we prove:

\begin{maintheorem}\label{main:regularity}
Let $S$ be a commutative semiring in which every multiplicative
idempotent is complemented.  Then the following are equivalent:
\begin{enumerate}[label=\upshape(\roman*)]
\item $S$ is von Neumann regular;
\item every ideal of $S$ is an intersection of maximal ideals
  containing it;
\item every ideal of $S$ is a $z$-ideal;
\item every principal ideal of $S$ is a $z$-ideal.
\end{enumerate}
\end{maintheorem}

The critical implication is (iv)$\Rightarrow$(i):
the hypothesis that every principal ideal is a $z$-ideal,
combined with the product formula $\M(a)=\M(a^2)$ valid in any
commutative semiring, forces the existence of the regular element
$x$ with $a=a^2x$.

The second and most broadly applicable result concerns the
pointfree structure of the $z$-ideal lattice.
By the coherent-frame spectrum theorem~\cite{Hochster,Johnstone},
a lattice is a coherent frame if and only if its prime spectrum
is a spectral space.
Unlike Theorem~A, the following theorem requires no hypothesis whatsoever
beyond the semiring axioms.

\begin{maintheorem}\label{main:zframe}
For every commutative semiring $S$, the ordered set $\ZId(S)$ of
$z$-ideals is a coherent frame.
Its compact elements are the finite joins of principal $z$-closures.
Consequently the space $\Spec_z(S)$ of prime $z$-ideals,
endowed with the hull-kernel topology, is spectral.
\end{maintheorem}

The key inputs are the product formula $\M(ab)=\M(a)\cup\M(b)$,
which holds because maximal ideals in commutative semirings are prime,
and the resulting meet identity $\ideal{a}_z\cap\ideal{b}_z=\ideal{ab}_z$
for principal $z$-closures.
These give the closure of compact elements under finite meets,
and hence coherence, without any auxiliary hypothesis.

The third result concerns the congruence-generated closure.
For each ideal $I$, the associated $g$-closed ideal
$I^g$ is the zero-class of the intersection of all maximal
congruences whose zero-classes contain $I$.
The map $I\mapsto I^g$ is always a closure operation, but the
finite-type and product-formula properties of Theorem~B hold
for $I^g$ only when the maximal congruences satisfy an explicit
compatibility condition.

\begin{maintheorem}\label{main:gframe}
Let $S$ satisfy the finite-type $g$-closure hypothesis stated in
Hypothesis~\ref{hyp:g-closure}.  Then the ordered set $\GCl(S)$
of $g$-closed ideals is a coherent frame.
Consequently the corresponding space $\Spec_g(S)$ of $g$-prime ideals
is spectral.  Moreover, the space of canonical prime $g$-congruences
is homeomorphic to $\Spec_g(S)$.
\end{maintheorem}

The contrast between Theorems~B and~C is deliberate.
Theorem~B holds unconditionally because the ordinary $z$-closure
is controlled by the maximal-ideal spectrum, which is primitive
enough that the product formula holds in every commutative semiring.
Theorem~C is conditional because the congruence-generated closure
draws on the full maximal-congruence spectrum, whose interaction
with the ideal theory requires additional hypotheses.
The natural-number semiring $\N$ is the canonical example
separating these two regimes.

Beyond the three main theorems, the paper develops the theory of
$z$-ideals under quotients and localizations
(Sections~\ref{sec:quotient}--\ref{sec:spectral-maps}),
identifying the precise saturation and lifting conditions required
for the $z$-ideal property to be preserved and reflected.
Section~\ref{sec:functors} provides a complete functorial formulation:
$\ZId$ and $\GCl$ become coherent-frame-valued functors on appropriate
categories of semirings, and the canonical closure map
$\sigma\colon\ZId\Rightarrow\GCl$ is a natural transformation under
an explicit generator-compatibility condition.

\medskip

\textbf{Organisation.}
Section~\ref{sec:prelim} fixes notation and recalls semiring ideals,
congruences, spectral spaces, coherent frames, and nuclei.
Section~\ref{sec:general} develops the maximal-hull calculus,
proves Theorem~A, and compares ordinary $z$-ideals with their
$z_k$-variants.
Section~\ref{sec:cong} introduces $z$-congruences and $g$-closed ideals,
with counterexamples showing why zero-classes of maximal congruences
require separate treatment.
Sections~\ref{sec:quotient} and~\ref{sec:local} treat quotients and
localizations, recording the saturation or lifting hypotheses needed
at each step.
Section~\ref{sec:spectral-maps} identifies the resulting spectra
topologically.
Section~\ref{sec:frame} proves Theorems~B and~C.
Section~\ref{sec:examples} collects model computations and
boundary examples.
Section~\ref{sec:functors} gives the functorial formulation.
The appendix gathers auxiliary structural lemmas whose proofs
are needed in the body but whose detail would interrupt the flow
of the main arguments.

\section{Preliminaries}\label{sec:prelim}

\subsection{Semirings, ideals, and congruences}

Throughout the paper, a \emph{semiring} means a commutative semiring
with $0$ and $1$, and all homomorphisms preserve $0$ and $1$.
Thus $(S,+,0)$ and $(S,\cdot,1)$ are commutative monoids,
multiplication distributes over addition, and $0s=0$ for all $s\in S$.
An \emph{ideal} $I$ of $S$ is a non-empty subset closed under
addition and under multiplication by arbitrary elements of $S$.
We write $\Id(S)$ for the set of all ideals.
Unless explicitly stated otherwise, prime ideals and maximal ideals
are proper.

The first obstruction to importing ring-theoretic proofs into semiring
theory is the failure of subtraction.
The following definition isolates the substitute for subtractivity
that is needed whenever ideals interact with quotients or congruences.

\begin{definition}[\cite{Golan/Book}]
An ideal $I$ of $S$ is a \emph{$k$-ideal}, or
\emph{subtractive ideal}, if $a+b\in I$ and $b\in I$ imply $a\in I$
for all $a,b\in S$.  It is \emph{strong} if $a+b\in I$ implies
$a\in I$ and $b\in I$ for all $a,b\in S$.
\end{definition}

For an ideal $I$, its $k$-closure is
\[
\mathcal C_k(I)=\{x\in S\mid x+y\in I\text{ for some }y\in I\}.
\]
This is the smallest $k$-ideal containing $I$.
The $k$-radical of $I$ is
\[
\mathcal R_k(I)=\bigcap\{P\in\Spec_k(S)\mid I\subseteq P\},
\]
where $\Spec_k(S)$ denotes the set of $k$-prime ideals.
We use standard facts on these two closures from~\cite{AG24}.

We shall repeatedly pass from maximal ideals to prime-ideal arguments.
The elementary lemma below identifies the precise point at which
commutativity and the existence of $1$ enter the maximal-hull calculus.

\begin{lemma}\label{lem:max-prime}
Every maximal ideal of a commutative semiring is prime.
\end{lemma}
\begin{proof}
Let $M$ be maximal and suppose $ab\in M$ with $a\notin M$.
Since the ideal generated by $M\cup\{a\}$ properly contains $M$,
it equals $S$.  Hence $1=m+sa$ for some $m\in M$ and $s\in S$.
Multiplying by $b$ gives $b=mb+sab\in M$.
\end{proof}

For rings, quotient objects are described entirely by ideals.
For semirings, quotients are governed by congruences, so we recall
the compatibility condition that replaces the ideal-kernel viewpoint.

\begin{definition}[\cite{Golan/Book}]
A \emph{congruence} $\Theta$ on a semiring $S$ is an equivalence
relation on $S$ compatible with both operations:
if $(a,b),(c,d)\in\Theta$, then
\[
(a+c,b+d)\in\Theta\qquad\text{and}\qquad(ac,bd)\in\Theta.
\]
\end{definition}

The set of all congruences on $S$ is $\Cong(S)$.
A congruence is \emph{proper} if it is not $S\times S$.
For an ideal $I$, the Bourne congruence associated to $I$ is
\[
k_I=\{(x,y)\in S\times S\mid x+a=y+b\text{ for some }a,b\in I\}.
\]
The \emph{zero-class} of a congruence $\Theta$ is
\[
\zero\Theta=\{a\in S\mid (a,0)\in\Theta\}.
\]
It is always a $k$-ideal.
A congruence is a \emph{$k$-congruence} if it equals $k_I$ for
some ideal $I$~\cite{Han-k-congruence}.

Prime congruences are the congruence-theoretic analogue of prime ideals,
formulated in terms of zero-class membership rather than elementwise
divisibility.

\begin{definition}[\cite{Han-k-congruence}]
A proper congruence $\Theta$ on $S$ is \emph{prime} if
$(xy,0)\in\Theta$ implies $(x,0)\in\Theta$ or $(y,0)\in\Theta$.
\end{definition}

The set of $k$-prime congruences is $\KCSpec(S)$.
We use the standard fact that maximal congruences are prime;
see~\cite[Theorem~5.5]{Han-k-congruence} for the corresponding
spectral statement for $k$-congruences.

\subsection{Spectral spaces and coherent frames}

A topological space $X$ is \emph{quasi-compact} if every open cover
has a finite subcover.
A non-empty closed subset $Y\subseteq X$ is \emph{irreducible} if
$Y=Y_1\cup Y_2$ with $Y_1,Y_2$ closed implies $Y=Y_1$ or $Y=Y_2$.
A point $y\in Y$ is a \emph{generic point} of $Y$ if
$Y=\overline{\{y\}}$.
The space $X$ is \emph{sober} if every irreducible closed subspace
has a unique generic point.

The topological target of our frame theorems is spectrality
in the following standard sense, which is the one most directly
connected with coherent frames and compact-open bases.

\begin{definition}[\cite{spectral-book,Hochster}]
A topological space $X$ is \emph{spectral} if it is quasi-compact
and sober and has a basis of quasi-compact open sets closed under
finite intersections.
\end{definition}

For a semiring $S$ and $a\in S$, put
\[
D(a)=\{P\in\Spec(S)\mid a\notin P\},\qquad
V(a)=\{P\in\Spec(S)\mid a\in P\}.
\]
Similarly define $D_k(a)$ and $V_k(a)$ on $\Spec_k(S)$.
The sets $D(a)$ form the standard Zariski basis on $\Spec(S)$.

The following theorem collects the spectral background already
available for semirings.
It provides the comparison point for the spectra of Theorems~B
and~C, which are built from $z$-closures rather than from
all prime ideals or all $k$-prime congruences.

\begin{theorem}[\cite{Pena,Han-k-congruence}]\label{thm:standard-spectral}
For a commutative semiring $S$ the following hold.
\begin{enumerate}
\item $\Spec(S)$, with the Zariski topology, is spectral.
\item $\Spec_k(S)$, with the Zariski topology, is spectral.
\item $\KCSpec(S)$ is spectral and homeomorphic to $\Spec_k(S)$.
\end{enumerate}
\end{theorem}

A \emph{frame} is a complete lattice $L$ in which finite meets
distribute over arbitrary joins:
$a\wedge\bigvee_{i\in\Lambda}b_i=\bigvee_{i\in\Lambda}(a\wedge b_i)$.
A frame homomorphism preserves arbitrary joins and finite meets.
An element $c\in L$ is \emph{compact} if $c\leqslant\bigvee A$ implies
$c\leqslant\bigvee A_0$ for some finite $A_0\subseteq A$.
A frame is \emph{coherent} if $1$ is compact, every element is a join
of compact elements, and compact elements are closed under finite meets.
The category of coherent frames with compact-element-preserving frame
homomorphisms is $\mathsf{CohFrm}$.

The frame proofs of Section~\ref{sec:frame} are most efficiently
expressed through closure operators on frames.
Prenuclei are the device that lets us pass from a manageable
preclosure to a full subframe.

\begin{definition}[\cite{Banaschewski88,Johnstone}]
A map $j_0:L\to L$ on a frame is a \emph{prenucleus} if,
for all $x,y\in L$,
\begin{enumerate}
\item $x\leqslant j_0(x)$;
\item $x\leqslant y$ implies $j_0(x)\leqslant j_0(y)$;
\item $j_0(x)\wedge y\leqslant j_0(x\wedge y)$.
\end{enumerate}
The \emph{nucleus} generated by $j_0$ is
\[
j(x)=\bigwedge\{y\in L\mid x\leqslant y=j_0(y)\}.
\]
\end{definition}

The following standard observation explains why our closure
constructions produce frames rather than merely complete lattices.

\begin{lemma}\label{lem:prenucleus-fixed}
The fixed sets of $j_0$ and of the nucleus $j$ generated by $j_0$
coincide.  This common fixed set is a frame.
\end{lemma}
\begin{proof}
If $x=j_0(x)$, then $x$ belongs to the defining meet; extensivity
gives $x\leqslant j(x)\leqslant x$, so $x=j(x)$.
Conversely, suppose $x=j(x)$.  Let $T=\{y:x\leqslant y=j_0(y)\}$.
For every $y\in T$, monotonicity gives $j_0(x)\leqslant j_0(y)=y$;
hence $j_0(x)\leqslant\bigwedge T=j(x)=x$, and extensivity gives
equality.  That the fixed points of a nucleus form a frame is
standard; see~\cite[Ch.~II, Lem.~2.2]{Johnstone}.
\end{proof}

The spectrum $\pt(L)$ of a frame $L$ is the set of frame homomorphisms
$L\to\two$; equivalently, the set of prime elements, with the
hull-kernel topology.
The coherent-frame spectrum theorem---the Stone duality that underlies
all spectral-space conclusions of this paper---states that coherent
frames and spectral spaces are contravariantly equivalent;
see~\cite[Ch.~II, Cor.~3.4]{Johnstone} and~\cite{Hochster}.

\section{General results on \texorpdfstring{$z$}{z}-ideals}
\label{sec:general}

Let $\Max(S)$ be the set of maximal ideals of $S$, and let
$\Maxk(S)$ be the set of maximal $k$-ideals.
For $a\in S$ define
\begin{equation}\label{eq:max-hulls}
\begin{array}{l@{\qquad}l}
\M(a)=\{M\in\Max(S)\mid a\in M\},&
\Mk(a)=\{M\in\Maxk(S)\mid a\in M\},\\[4pt]
\mfrak(a)=\displaystyle\bigcap_{M\in\M(a)}M,&
\mfrak_k(a)=\displaystyle\bigcap_{M\in\Mk(a)}M.
\end{array}
\end{equation}
The intersection over an empty family is $S$;
this convention is necessary because a unit belongs to no
maximal ideal.

\subsection{The maximal-hull calculus}

The $z$-ideal theory rests on a product formula and a
sum-inclusion for maximal hulls.
The product formula is the semiring analogue of the equality of
maximal vanishing sets, and it requires only that maximal ideals
be prime, which Lemma~\ref{lem:max-prime} guarantees.
The sum statement is one-sided because subtraction is unavailable.

\begin{lemma}\label{loc6}
For all $a,b\in S$,
\begin{enumerate}
\item $\M(ab)=\M(a)\cup\M(b)$;
\item $\M(a)\cap\M(b)\subseteq\M(a+b)$.
\end{enumerate}
\end{lemma}
\begin{proof}
The first assertion follows from Lemma~\ref{lem:max-prime}:
a maximal ideal contains $ab$ if and only if it contains $a$ or $b$.
The second is immediate from closure of ideals under addition.
\end{proof}

Intersections of maximal ideals over a given element are the
algebraic closures behind $z$-ideals.
The next proposition translates the hull identities
into formulas for these intersections.

\begin{proposition}\label{prop:mfrak-basic}
For all $a,b\in S$,
\begin{enumerate}
\item $\mfrak(ab)=\mfrak(a)\cap\mfrak(b)$;
\item $\mfrak(a+b)\subseteq\displaystyle
  \bigcap\{M\in\Max(S)\mid a,b\in M\}$.
\end{enumerate}
\end{proposition}
\begin{proof}
By Lemma~\ref{loc6}, the maximal ideals containing $ab$ are exactly
those containing $a$ or $b$.  Intersecting over this union gives
\[
\mfrak(ab)=\bigcap_{M\in\M(a)\cup\M(b)}M
=\Bigl(\bigcap_{M\in\M(a)}M\Bigr)\cap\Bigl(\bigcap_{M\in\M(b)}M\Bigr).
\]
For (2), Lemma~\ref{loc6} gives $\M(a)\cap\M(b)\subseteq\M(a+b)$;
taking intersections reverses inclusions.
\end{proof}

\subsection{\texorpdfstring{$z$}{z}-ideals and their closure}

An ideal is $z$-closed if membership is stable under the
maximal-hull equivalence: once an element of the ideal and a second
element lie in exactly the same maximal ideals, the second element
must also belong to the ideal.

\begin{definition}[\cite{BiswasFilomat,AG23}]\label{def:zideal}
An ideal $I$ of $S$ is a \emph{$z$-ideal} if,
whenever $a\in I$ and $b\in S$ satisfy $\M(a)=\M(b)$, one has
$b\in I$.  The set of $z$-ideals of $S$ is $\ZId(S)$.
\end{definition}

For an ideal $I$ define its \emph{$z$-closure}
\[
\clz(I)=\bigcap\{J\in\ZId(S)\mid I\subseteq J\}.
\]
Then $I$ is a $z$-ideal precisely when $\clz(I)=I$.
We write $\ideal{a}_z$ for $\clz(\ideal{a})$,
the principal $z$-closure of $a$.
The canonical description $\ideal{a}_z=\mfrak(a)$---proved in
the appendix as Lemma~\ref{app:maximal-hull-calculus}---shows
that the principal $z$-closure is simply the intersection of
all maximal ideals containing $a$.

The natural-number semiring is the canonical test case:
its maximal-ideal space is minimal, making explicit computations
transparent, while the ideal structure exhibits the full range of
phenomena that arise for semirings but not for rings.

\begin{example}\label{ex:N-zideals}
In $S=\N$, the unique ordinary maximal ideal is $\N\setminus\{1\}$.
Hence every non-unit has the same maximal hull, and the only proper
$z$-ideal of $\N$ is $\N\setminus\{1\}$:
any proper $z$-ideal must contain every non-unit
(because $\M(0)=\M(n)$ for every $n\neq1$) and cannot contain $1$.
\end{example}

The following result explains when, in additively idempotent semirings,
$z$-ideals are automatically subtractive.
The characterization is phrased entirely in terms of maximal hulls
of sums, making explicit where lower-set behaviour enters.

\begin{proposition}\label{prop:idempotent-z-k}
Suppose $S$ is additively idempotent.  Then every $z$-ideal is
a $k$-ideal if and only if $\M(a+b)=\M(a)\cap\M(b)$ for all $a,b\in S$.
\end{proposition}
\begin{proof}
Assume first that every $z$-ideal is a $k$-ideal.
Every maximal ideal is a $z$-ideal: if $a\in M$ and $\M(a)=\M(b)$,
then $M\in\M(b)$ and hence $b\in M$.
Thus every maximal ideal is a $k$-ideal.
If $a+b\in M$, then in the natural order of the idempotent semiring,
$a\leqslant a+b$ and $b\leqslant a+b$;
the $k$-property gives $a,b\in M$.
Hence $\M(a+b)\subseteq\M(a)\cap\M(b)$,
and the reverse inclusion is Lemma~\ref{loc6}.

Conversely, assume the displayed equality.
In an additively idempotent semiring, a $k$-ideal is the same as a
lower set for the natural order.
Suppose $b\leqslant a$ and $a\in I$ for a $z$-ideal $I$.
Then $a+b=a$, so
\[
\M(a)=\M(a+b)=\M(a)\cap\M(b),
\]
whence $\M(a)\subseteq\M(b)$.
Put $c=ab$; Lemma~\ref{loc6} gives $\M(c)=\M(b)$.
Since $c\in I$ and $I$ is a $z$-ideal, $b\in I$.
\end{proof}

Colon ideals provide a compact way to test the stability of an ideal
under replacement of an element by its powers.
The following is the semiring version of the familiar fact that
$z$-ideals do not distinguish between an element and its powers
at the level of maximal hulls.

\begin{proposition}\label{prop:colon}
If $I$ is a $z$-ideal, then $(I:x)=(I:x^n)$
for every $x\in S$ and every integer $n\geqslant1$.
\end{proposition}
\begin{proof}
The inclusion $(I:x)\subseteq(I:x^n)$ is clear.
If $b\in(I:x^n)$, then $x^nb\in I$.
By Lemma~\ref{loc6}, $\M(xb)=\M(x^n b)$.
Since $I$ is a $z$-ideal, $xb\in I$, hence $b\in(I:x)$.
\end{proof}

Prime objects in the $z$-ideal frame are tested against products
of $z$-ideals, not merely against products of elements.
The next definition fixes this convention and introduces the
corresponding $z$-radical.

\begin{definition}[\cite{AG23}]
A proper $z$-ideal $P$ is \emph{$z$-prime} if $AB\subseteq P$
implies $A\subseteq P$ or $B\subseteq P$ for all $A,B\in\ZId(S)$.
The \emph{$z$-radical} of a $z$-ideal $I$ is
\[
\sqrt[z]{I}=\bigcap\{P\in\Spec_z(S)\mid I\subseteq P\}.
\]
\end{definition}

The exchange theorem below is the bridge between the frame-theoretic
language of $z$-prime ideals and the more familiar ideal-theoretic
language of prime ideals.

\begin{theorem}[\cite{AG25}]\label{thm:zprime-exchange}
An ideal $P$ is $z$-prime if and only if $P$ is both prime and
a $z$-ideal.
\end{theorem}

The next proposition shows that principal maximal hulls can also
be read from $z$-prime data.
This is the precise form in which the hull-kernel topology
re-emerges in the prime spectrum of the $z$-ideal frame.

\begin{proposition}\label{prop:principal-radical-hulls}
For $a,b\in S$,
\[
\M(a)=\M(b)\quad\Longleftrightarrow\quad
\sqrt[z]{\ideal{a}_z}=\sqrt[z]{\ideal{b}_z}.
\]
Equivalently, $\M(a)=\M(b)$ if and only if $a$ and $b$ belong to
exactly the same $z$-prime ideals.
\end{proposition}
\begin{proof}
Assume $\M(a)=\M(b)$.
If $P$ is a $z$-prime ideal containing $a$, then $P$ is a $z$-ideal,
so the equality of maximal hulls forces $b\in P$.
By symmetry, the $z$-prime ideals containing $a$ and $b$ coincide,
giving equal $z$-radicals.

Conversely, assume the $z$-radicals are equal.
Let $M\in\M(a)$.  Then $M$ is maximal, hence prime by
Lemma~\ref{lem:max-prime}; the argument in the proof of
Proposition~\ref{prop:idempotent-z-k} shows that $M$ is a $z$-ideal.
Thus $M$ is $z$-prime and contains $a$,
hence contains $\sqrt[z]{\ideal{a}_z}=\sqrt[z]{\ideal{b}_z}$,
and therefore contains $b$.
This gives $\M(a)\subseteq\M(b)$; the reverse inclusion is symmetric.
\end{proof}

\subsection{Zero-dimensionality, von Neumann regularity, and
\texorpdfstring{$z$}{z}-ideals}

The Krull dimension $\dim S$ is the supremum of the lengths of chains
of prime ideals.  A semiring is \emph{zero-dimensional} if $\dim S=0$.
Semifields are zero-dimensional, but the semiring situation is subtler
than the ring situation, as the following remark illustrates.

\begin{remark}
The semiring $\N$ has the chain
$\ideal{0}\subsetneq p\N\subsetneq\N\setminus\{1\}$,
so its prime-ideal structure differs markedly from that of $\mathds{Z}$.
A bounded distributive lattice $L$, viewed as an idempotent semiring
via $a+b=a\vee b$ and $ab=a\wedge b$, is von Neumann regular because
$a=a^2$ for all $a\in L$.
Yet zero-dimensionality of $L$ is equivalent to $L$ being Boolean;
see~\cite{Nachbin47}.
For instance, in the distributive lattice
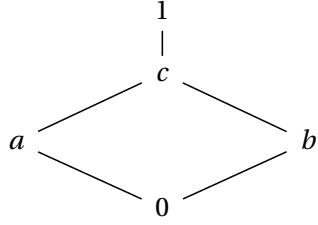
\[
\begin{tikzcd}[row sep=small,column sep=large]
&1\arrow[d,dash]&\\
&c\arrow[dl,dash]\arrow[dr,dash]&\\
a\arrow[dr,dash]&&b\arrow[dl,dash]\\
&0&
\end{tikzcd}
\]
the ideal $\{0,a\}$ is prime but not maximal.
\end{remark}

The following definition is the standard semiring version of von
Neumann regularity, due to Subramanian~\cite{Subramanian70/RegularSemirings}.

\begin{definition}[\cite{Golan/Book,Subramanian70/RegularSemirings}]
A semiring $S$ is \emph{von Neumann regular} if for every $a\in S$
there exists $x\in S$ such that $a=a^2x$.
\end{definition}

Let $E(S)=\{e\in S\mid e^2=e\}$ be the set of multiplicative
idempotents.  An idempotent $e$ is \emph{complemented} if there is
an idempotent $f$ with $ef=0$ and $e+f=1$.
The set of complemented idempotents is $\mathrm{Comp}(S)$.

The phrase ``strong $z$-ideal'' has been used in more than one way
in the literature.  To avoid ambiguity, we reserve it here for ideals
completely determined by the maximal ideals above them---the semiring
analogue of the structure present in von Neumann regular rings.

\begin{definition}\label{def:strong-z}
An ideal $I$ is a \emph{strong $z$-ideal} in this paper if
\[
I=\bigcap\{M\in\Max(S)\mid I\subseteq M\}.
\]
\end{definition}

We can now prove the regularity theorem announced in the introduction.
The proof follows the maximal-hull strategy from the ring case~\cite{Mason71/z-ideals/prime-ideals},
but the complemented-idempotent hypothesis is what makes the
separation argument work without subtraction.

\begin{theorem}\label{thm:regularity}
Let $S$ be a semiring such that $E(S)=\mathrm{Comp}(S)$.
Then the following are equivalent.
\begin{enumerate}
\item $S$ is von Neumann regular.
\item Every ideal of $S$ is a strong $z$-ideal.
\item Every ideal of $S$ is a $z$-ideal.
\item Every principal ideal of $S$ is a $z$-ideal.
\end{enumerate}
\end{theorem}
\begin{proof}
(1)$\Rightarrow$(2).
Let $I$ be an ideal and let $a\notin I$.
Choose $x\in S$ with $a=a^2x$ and put $e=ax$.
Then $e^2=e$ and $\ideal{a}=\ideal{e}$:
indeed $e=ax\in\ideal{a}$ and $a=ae\in\ideal{e}$.
Since $a\notin I$, also $e\notin I$.
By hypothesis, $e$ has a complement $f$
satisfying $ef=0$ and $e+f=1$.

Consider $J=I+\ideal{f}$.
If $1\in J$, write $1=i+sf$ for some $i\in I$ and $s\in S$,
and multiply by $e$ to get $e=ei\in I$, a contradiction.
Hence $J$ is proper, so some maximal ideal $M$ contains $J$.
Since $f\in M$ and $e+f=1$, we have $e\notin M$, hence $a\notin M$.
Thus every element outside $I$ is excluded by a maximal ideal
containing $I$, proving (2).

(2)$\Rightarrow$(3).
Let $I$ be a strong $z$-ideal, $a\in I$, and $\M(a)=\M(b)$.
Every maximal ideal $M$ containing $I$ contains $a$,
hence $M\in\M(a)=\M(b)$, so $b\in M$.
Intersecting over all such $M$ gives $b\in I$.

(3)$\Rightarrow$(4) is immediate.

(4)$\Rightarrow$(1).
Let $a\in S$.  Since maximal ideals are prime,
Lemma~\ref{loc6}(1) gives $\M(a)=\M(a^2)$.
The ideal $\ideal{a^2}$ is a $z$-ideal by hypothesis, and
$a^2\in\ideal{a^2}$; hence $a\in\ideal{a^2}$,
giving $a=a^2x$ for some $x$.
\end{proof}

\subsection{Variants of \texorpdfstring{$z$}{z}-ideals}

The ordinary maximal spectrum $\Max(S)$ and the maximal $k$-spectrum
$\Maxk(S)$ need not coincide.
The following variant records the parallel $z$-closure obtained
by testing against maximal $k$-ideals instead of all maximal ideals.

\begin{definition}\label{def:zk}
An ideal $I$ is a \emph{$z_k$-ideal} if $a\in I$ and $\Mk(a)=\Mk(b)$
imply $b\in I$.
\end{definition}

When the two maximal spectra agree as subsets of $\Id(S)$,
the two closure theories coincide.

\begin{proposition}\label{prop:zk-safe}
Suppose $\Max(S)=\Maxk(S)$ as subsets of $\Id(S)$.
Then the $z$-ideals and the $z_k$-ideals of $S$ coincide.
In particular, this holds in any class of semirings
for which all maximal ideals are $k$-ideals and every
maximal $k$-ideal is maximal as an ordinary ideal.
\end{proposition}
\begin{proof}
Under the stated hypothesis $\M(a)=\Mk(a)$ for every $a\in S$,
so the two definitions are identical.
\end{proof}

The natural-number semiring again shows that this hypothesis is
not merely terminological: ordinary maximal ideals and maximal
$k$-ideals lead to genuinely different closure theories.

\begin{example}\label{ex:zk-N}
For $S=\N$, the unique proper $z$-ideal is $\N\setminus\{1\}$
(Example~\ref{ex:N-zideals}), whereas the maximal $k$-ideals are
$p\N$ for primes $p$.
Hence $n\N$ is a $z_k$-ideal whenever $n$ is square-free,
because membership in the maximal $k$-ideals records the prime
divisors of an element.
This shows that $z$-ideals and $z_k$-ideals can differ substantially.
\end{example}

Strong maximal ideals control additive decompositions within a single
maximal ideal.
The Gelfand condition packages this behaviour in a semiring-wide
property.

\begin{definition}[\cite{Golan/Book}]
A semiring $S$ is a \emph{Gelfand semiring} if $u+s$ is a unit
for every unit $u\in S$ and every $s\in S$.
\end{definition}

\begin{proposition}\label{prop:gelfand-strong-max}
A semiring $S$ is Gelfand if and only if every maximal ideal
of $S$ is strong.
\end{proposition}
\begin{proof}
Assume $S$ is Gelfand and let $M$ be maximal.
If $a+b\in M$ and $a\notin M$, then $M+\ideal{a}=S$,
so $1=m+ra$ for some $m\in M$.
Then $ra+rb\in M$ and $1+rb=m+ra+rb\in M$.
But $1+rb$ is a unit, a contradiction, so $a\in M$,
and similarly $b\in M$.

Conversely, if every maximal ideal is strong and $u$ is a unit,
suppose $u+s$ lies in some maximal ideal $M$.
Strength of $M$ gives $u\in M$, which is impossible.
\end{proof}

Irreducibility and strong irreducibility organize the prime side of
the $z$-ideal theory.
We recall these notions in the $z$-closed lattice because the
quotient and localization results of the next two sections
use them to state spectral equivalences.

\begin{definition}[\cite{AG23}]
A $z$-ideal $I$ is \emph{$z$-irreducible} if $A\cap B=I$ with
$A,B\in\ZId(S)$ implies $A=I$ or $B=I$.
It is \emph{$z$-strongly irreducible} if $A\cap B\subseteq I$
implies $A\subseteq I$ or $B\subseteq I$ for all $A,B\in\ZId(S)$.
\end{definition}

The complement of a prime object is best described multiplicatively.
The $i_z$-systems below are the $z$-ideal analogue of the
multiplicatively closed sets used to detect prime ideals.

\begin{definition}\label{def:iz-system}
A subset $B\subseteq S$ is an \emph{$i_z$-system} if $a,b\in B$ imply
\[
\ideal{a}_z\cap\ideal{b}_z\cap B\neq\varnothing.
\]
\end{definition}

The following characterization translates $z$-strong irreducibility
into a statement about products inside an $i_z$-system,
which is the form needed for the separation arguments later.

\begin{proposition}\label{loc14}
For a $z$-ideal $I$ of $S$, the following are equivalent.
\begin{enumerate}
\item $I$ is $z$-strongly irreducible.
\item If $\ideal{a}_z\cap\ideal{b}_z\subseteq I$, then $a\in I$ or $b\in I$.
\item $S\setminus I$ is an $i_z$-system.
\end{enumerate}
\end{proposition}
\begin{proof}
(1)$\Rightarrow$(2) is immediate since $\ideal{a}_z$ and $\ideal{b}_z$
are $z$-ideals.
If (2) holds and $a,b\in S\setminus I$, then
$\ideal{a}_z\cap\ideal{b}_z\not\subseteq I$,
so it meets $S\setminus I$; hence (3).
Conversely, assume (3) and suppose $A\cap B\subseteq I$ for
$A,B\in\ZId(S)$ with $A\not\subseteq I$ and $B\not\subseteq I$.
Choose $a\in A\setminus I$ and $b\in B\setminus I$.
Since $A$ and $B$ are $z$-ideals,
$\ideal{a}_z\subseteq A$ and $\ideal{b}_z\subseteq B$.
The $i_z$-system condition gives an element of $(A\cap B)\setminus I$,
a contradiction.
\end{proof}

Strong irreducibility alone does not imply primality:
the radical condition below is what bridges the gap.

\begin{proposition}\label{prop:zstrong-prime-radical}
Let $I$ be a proper $z$-strongly irreducible ideal.
Then $I$ is $z$-prime if and only if $I=\sqrt[z]{I}$.
\end{proposition}
\begin{proof}
If $I$ is $z$-prime, it appears in its own radical and equality holds.
Conversely, assume $I=\sqrt[z]{I}$ and let $A,B\in\ZId(S)$ with
$AB\subseteq I$.
By the radical product formula for $z$-radicals~\cite{AG25},
$\sqrt[z]{A\cap B}=\sqrt[z]{AB}$, so
\[
A\cap B\subseteq\sqrt[z]{A\cap B}=\sqrt[z]{AB}\subseteq\sqrt[z]{I}=I.
\]
Strong irreducibility then gives $A\subseteq I$ or $B\subseteq I$.
\end{proof}

Combining the preceding proposition with Theorem~\ref{thm:zprime-exchange}
gives an ordinary ideal-theoretic criterion for $z$-primality.
This corollary is often the most direct route to recognizing
$z$-prime ideals in concrete examples.

\begin{corollary}\label{cor:ordinary-prime}
For a proper $z$-ideal $I$, the following are equivalent:
\begin{enumerate}
\item $I$ is an ordinary prime ideal;
\item $I$ is $z$-prime;
\item $AB\subseteq I$ implies $A\subseteq I$ or $B\subseteq I$
  for all $A,B\in\ZId(S)$.
\end{enumerate}
\end{corollary}
\begin{proof}
This is Theorem~\ref{thm:zprime-exchange} together with the
definition of $z$-prime.
\end{proof}

\section{\texorpdfstring{$z$}{z}-congruences on semirings}\label{sec:cong}

Since quotient objects in semiring theory are governed by congruences
rather than by ideals, it is natural to develop a congruence-theoretic
analogue of the $z$-ideal notion.
This section introduces $z$-congruences and the canonical
$g$-closure, and shows through explicit examples that the two
closures---ideal-theoretic and congruence-theoretic---diverge
already in the natural-number semiring.

Let $\MaxC(S)$ denote the set of maximal congruences on $S$.
For $(a,b)\in S\times S$ put
\[
\MC(a,b)=\{\theta\in\MaxC(S)\mid (a,b)\in\theta\}.
\]
For a congruence $\rho$ and an ideal $I$, define the
maximal-congruence hulls
\[
\MC(\rho)=\{\theta\in\MaxC(S)\mid \rho\subseteq\theta\},
\qquad
\MC(I)=\{\theta\in\MaxC(S)\mid I\subseteq\zero\theta\}.
\]
These two hulls must not be confused: $\MC(\rho)\subseteq\MC(\zero\rho)$
always, but equality may fail.

\subsection{\texorpdfstring{$z$}{z}-congruences}

The closure form in the following definition is intentional:
it is stable under arbitrary intersections,
whereas a formulation using only equality of maximal-congruence hulls
would not support the closure construction that follows.

\begin{definition}\label{def:zcong}
A congruence $\rho$ on $S$ is a \emph{$z$-congruence} if
\[
(a,b)\in\rho\quad\text{and}\quad\MC(a,b)\subseteq\MC(c,d)
\quad\Longrightarrow\quad(c,d)\in\rho.
\]
The set of all $z$-congruences is $\ZCong(S)$.
\end{definition}

The inclusion form implies the stronger equality formulation whenever
equality of maximal-congruence hulls is known,
and it is the form needed for the quotient arguments
of Section~\ref{sec:quotient}.
The definition has natural equivalent reformulations.

\begin{proposition}\label{prop:zcong-equivalences}
For a congruence $\rho$ on $S$, the following are equivalent.
\begin{enumerate}
\item $\rho$ is a $z$-congruence.
\item If $(a,b)\in\rho$ and $(c,d)\notin\rho$, then there exists
  $\theta\in\MC(a,b)$ with $\theta\notin\MC(c,d)$.
\item For every $(a,b)\in\rho$,
  $\displaystyle\bigcap_{\theta\in\MC(a,b)}\theta\subseteq\rho$.
\end{enumerate}
\end{proposition}
\begin{proof}
(1)$\Rightarrow$(2) is the contrapositive of the definition.
If (2) holds and $(c,d)$ belongs to every maximal congruence
containing $(a,b)$, then no separating congruence as in (2) exists,
so $(c,d)\in\rho$; this gives (3).
If (3) holds and $\MC(a,b)\subseteq\MC(c,d)$ with $(a,b)\in\rho$,
then $(c,d)$ belongs to every maximal congruence containing $(a,b)$,
so $(c,d)\in\rho$.
\end{proof}

The reason for using the closure formulation is visible in the next
proposition: arbitrary intersections of $z$-congruences exist,
so every congruence has a smallest $z$-congruence above it.

\begin{proposition}\label{prop:zcong-intersections}
Arbitrary intersections of $z$-congruences are $z$-congruences.
In particular, every congruence $\rho$ has a smallest
$z$-congruence containing it, namely
\[
\clc(\rho)=\bigcap\{\sigma\in\ZCong(S)\mid\rho\subseteq\sigma\}.
\]
\end{proposition}
\begin{proof}
Let $\{\rho_\lambda\}$ be a family of $z$-congruences and
$\rho=\bigcap\rho_\lambda$.
If $(a,b)\in\rho$ and $\MC(a,b)\subseteq\MC(c,d)$,
then $(a,b)\in\rho_\lambda$ for each $\lambda$,
hence $(c,d)\in\rho_\lambda$, and therefore $(c,d)\in\rho$.
\end{proof}

Maximal congruences provide the basic closed points of the
congruence picture.
The next result confirms that they are already $z$-congruences,
ensuring that the $g$-closure construction is compatible with the
maximal-congruence hull.

\begin{proposition}\label{prop:max-zcong}
Every maximal congruence is a $z$-congruence.  More generally,
every maximal proper $z$-congruence is meet-irreducible in
$\ZCong(S)$.
\end{proposition}
\begin{proof}
If $\theta$ is maximal and $(a,b)\in\theta$ with
$\MC(a,b)\subseteq\MC(c,d)$, then $\theta\in\MC(a,b)\subseteq\MC(c,d)$,
so $(c,d)\in\theta$.
For the second assertion, let $\rho$ be maximal among proper
$z$-congruences and suppose $\rho=\rho_1\cap\rho_2$ in $\ZCong(S)$.
Then $\rho_i=\rho$ or $\rho_i=S\times S$, giving
$\rho=\rho_1$ or $\rho=\rho_2$.
\end{proof}

\subsection{Examples of \texorpdfstring{$z$}{z}-congruences}

The congruences of $\N$ give a transparent arithmetic model for
the preceding definitions.
They show that maximal-congruence hulls record the prime divisors
of moduli rather than just the ordinary maximal ideal of $\N$.

\begin{example}\label{ex:N-congruences}
Let $S=\N$.
For $d\geqslant2$, let $k_d$ be congruence modulo $d$:
$(a,b)\in k_d\Leftrightarrow a\equiv b\pmod{d}$.
There is also the Rees congruence
$\beta=(\N_{>0}\times\N_{>0})\cup\{(0,0)\}$,
which collapses all positive integers.
The maximal congruences on $\N$ are $\beta$ and the congruences
$k_p$ with $p$ prime:
cancellative congruences on $\N$ are exactly the $k_d$,
and maximality forces $d$ to be prime;
the remaining maximal congruence is $\beta$.

For positive $a,b$, the maximal-congruence hull is
\[
\MC(a,b)=\{\beta\}\cup\{k_p\mid p\text{ prime and }p\mid a-b\}.
\]
Consequently, $k_d$ is a $z$-congruence if and only if $d$ is
square-free.
For instance, $k_{12}$ is not a $z$-congruence:
$(15,27)\in k_{12}$ and
$\MC(15,27)=\{\beta,k_2,k_3\}=\MC(18,36)$,
but $(18,36)\notin k_{12}$.
\end{example}

Boolean power-set semirings are the contrasting model.
They illustrate a setting where congruences, ideals, and maximal-hull
arguments align particularly cleanly.

\begin{example}\label{ex:powerset}
Let $X$ be finite and consider the Boolean semiring
$(\mathcal{P}(X),\cup,\cap,\varnothing,X)$.
For $x\in X$ define $(A,B)\in\Theta_x\Leftrightarrow[x\in A
\text{ iff }x\in B]$.
The congruences $\Theta_x$ are the maximal congruences.
Every congruence on $\mathcal{P}(X)$ is determined by an ideal
of the Boolean algebra $\mathcal{P}(X)$ under symmetric difference,
and the maximal-congruence hull records the coordinates on which
two subsets differ.
Hence every congruence on a finite power-set semiring is a
$z$-congruence.
\end{example}

\subsection{The canonical \texorpdfstring{$g$}{g}-closure
and \texorpdfstring{$g$}{g}-closed ideals}

For an ideal $I$ define
\begin{equation}\label{eq:wpI}
\wpI{I}=\{(a,b)\in S\times S\mid\MC(I)\subseteq\MC(a,b)\}
       =\bigcap_{\theta\in\MC(I)}\theta,
\end{equation}
where the intersection over an empty family is $S\times S$.
The most natural congruence associated to an ideal is obtained by
intersecting the maximal congruences whose zero-classes contain
the ideal; the next proposition shows this intersection lands in
$\ZCong(S)$.

\begin{proposition}\label{prop:wpI}
For every ideal $I$ of $S$, the relation $\wpI{I}$ is a $z$-congruence.
\end{proposition}
\begin{proof}
It is an intersection of congruences, hence a congruence.
If $(a,b)\in\wpI{I}$ and $\MC(a,b)\subseteq\MC(c,d)$,
then $\MC(I)\subseteq\MC(c,d)$, so $(c,d)\in\wpI{I}$.
\end{proof}

The zero-class of the associated congruence is the ideal-theoretic
shadow of the congruence construction.
We call its fixed points $g$-closed ideals to distinguish them
from ordinary $z$-ideals; this distinction is not merely nominal,
as the example below will make precise.

\begin{definition}\label{def:g-closed-ideals}
The \emph{canonical $g$-closure} of an ideal $I$ is
\[
I^g=\zero{\wpI{I}}=\{a\in S\mid\MC(I)\subseteq\MC(a,0)\}.
\]
An ideal $I$ is \emph{$g$-closed} if $I=I^g$.
We write $\GCl(S)$ for the set of all $g$-closed ideals of $S$.
\end{definition}

\begin{proposition}\label{prop:g-closure-operator}
The assignment $g:\Id(S)\to\Id(S)$, $g(I)=I^g$,
is an extensive, order-preserving, idempotent closure operation
whose fixed points are precisely the $g$-closed ideals.
\end{proposition}
\begin{proof}
Extensiveness follows from $I\subseteq\zero{\wp_I}$.
If $I\subseteq J$, then $\MC(J)\subseteq\MC(I)$,
hence $\wp_I\subseteq\wp_J$ and $I^g\subseteq J^g$.
For idempotence, observe that $\MC(I^g)=\MC(I)$:
the inclusion $I\subseteq I^g$ gives $\MC(I^g)\subseteq\MC(I)$,
while for each $\theta\in\MC(I)$ one has $\wp_I\subseteq\theta$
and hence $I^g=\zero{\wp_I}\subseteq\zero\theta$,
giving $\theta\in\MC(I^g)$.
Therefore $\wp_{I^g}=\wp_I$ and $(I^g)^g=I^g$.
\end{proof}

The following example records one of the basic distinctions that
recurs throughout the paper.
It shows concretely that a maximal congruence can have a zero-class
that is neither a maximal ideal nor a $z$-ideal.

\begin{example}\label{ex:zero-class-distinction}
In $\N$, the maximal congruence $\beta$ has zero-class $\{0\}$.
This zero-class is not a $z$-ideal:
$\M(0)=\M(2)$ but $2\notin\{0\}$.
Nevertheless $\{0\}$ is $g$-closed:
since $\beta\in\MC(0)$ and $\zero\beta=\{0\}$,
no positive element lies in $0_{\wp_{\{0\}}}$.

On the other hand, the ordinary maximal ideal $\N\setminus\{1\}$
is not $g$-closed.
No maximal congruence has zero-class containing $\N\setminus\{1\}$,
so $\MC(\N\setminus\{1\})=\varnothing$,
whence $\wp_{\N\setminus\{1\}}=\N\times\N$
and $(\N\setminus\{1\})^g=\N$.
\end{example}

\begin{remark}\label{rem:zero-class-distinction}
Example~\ref{ex:zero-class-distinction} isolates two structural
limitations that recur throughout the paper:
zero-classes of maximal congruences need not be maximal ideals,
and $g$-closed ideals need not be ordinary $z$-ideals.
Subsequent quotient and spectral statements include the
hypotheses needed to navigate this separation.
\end{remark}

The next proposition records the congruence-theoretic analogues
of the maximal-ideal-hull identities from Section~\ref{sec:general}.
These are the properties that drive the proof of Theorem~C.

\begin{proposition}\label{loc24}
For ideals $I,J$ of $S$,
\begin{enumerate}
\item if $I\subseteq J$, then $\MC(J)\subseteq\MC(I)$;
\item $\MC(I+J)=\MC(I)\cap\MC(J)$;
\item if maximal congruences have prime zero-classes, then
$\MC(IJ)=\MC(I\cap J)=\MC(I)\cup\MC(J)$.
\end{enumerate}
\end{proposition}
\begin{proof}
Parts (1) and (2) are immediate from the definition.
For (3), let $\theta$ be a maximal congruence with
$\zero\theta$ prime.
Containment of $I$ or $J$ in $\zero\theta$ implies containment
of both $IJ$ and $I\cap J$.
Conversely, if $I\cap J\subseteq\zero\theta$ and neither $I$ nor $J$
is contained in $\zero\theta$, choose $a\in I\setminus\zero\theta$
and $b\in J\setminus\zero\theta$;
then $ab\in IJ\subseteq I\cap J\subseteq\zero\theta$,
contradicting primeness.
The same argument applies with $IJ$ in place of $I\cap J$.
\end{proof}

Finite meets of canonical associated congruences require
a prime-zero-class hypothesis.
The following proposition states the exact condition under which
products of ideals control intersections of associated congruences.

\begin{proposition}\label{prop:wp-meets}
Assume maximal congruences have prime zero-classes.
Then for all ideals $I,J$,
\[
\wp_{I\cap J}=\wp_{IJ}=\wp_I\cap\wp_J,
\]
and moreover $\wp_I\cup\wp_J\subseteq\wp_{I+J}$.
\end{proposition}
\begin{proof}
By Proposition~\ref{loc24}(3),
$\MC(I\cap J)=\MC(IJ)=\MC(I)\cup\MC(J)$.
Intersecting the maximal congruences in this union gives
$\wp_I\cap\wp_J$.
For the inclusion, if $(a,b)\in\wp_I$, then every maximal congruence
whose zero-class contains $I+J$ has zero-class containing $I$,
hence contains $(a,b)$, giving $(a,b)\in\wp_{I+J}$;
the same argument applies starting from $\wp_J$.
\end{proof}

The ideal and congruence sides will be compared only for canonical
objects, since the zero-class map has a well-defined inverse only
on that restricted class.

\begin{definition}\label{def:g-cong}
A congruence $\rho$ is a \emph{canonical $g$-congruence} if
$\rho=\wp_I$ for some $g$-closed ideal $I$.
Equivalently, $\rho=\wp_{\zero\rho}$ and $\zero\rho$ is $g$-closed.
\end{definition}

The following proposition gives the order correspondence between
$g$-closed ideals and their associated congruences;
it is the algebraic foundation for the homeomorphism in Theorem~C.

\begin{proposition}\label{prop:g-correspondence-basic}
The assignments $I\mapsto\wp_I$ and $\rho\mapsto\zero\rho$ restrict
to mutually inverse order-preserving bijections between
$g$-closed ideals and canonical $g$-congruences.
\end{proposition}
\begin{proof}
If $I$ is $g$-closed, then $\zero{\wp_I}=I$ by definition.
If $\rho=\wp_I$ is canonical, then $\wp_{\zero\rho}=\wp_I=\rho$.
Monotonicity is direct.
\end{proof}

\section{Quotient semirings with respect to associated
\texorpdfstring{$z$}{z}-congruences}\label{sec:quotient}

Let $I$ be an ideal of $S$, and let
\[
\pi_I:S\longrightarrow S_I:=S/\wp_I
\]
be the quotient homomorphism.
We write $[a]_I$ for the class of $a$ modulo $\wp_I$.
This notation deliberately avoids identifying the zero-class of
the quotient with $I$; in general that zero-class is $I^g$,
and it equals $I$ if and only if $I$ is $g$-closed.

Quotients by associated $z$-congruences must be treated through
congruence classes rather than through ordinary quotient ideals.
The following proposition records the elementary facts needed before
discussing ideal correspondences.

\begin{proposition}\label{prop:quotient-basic}
For every ideal $I$ of $S$:
\begin{enumerate}
\item $[a]_I=[0]_I$ if and only if $a\in I^g$;
\item in particular, if $a\in I$ then $[a]_I=[0]_I$;
\item $[a]_I=[0]_I$ implies $a\in I$ for all $a$ if and only if
  $I$ is $g$-closed;
\item for every ideal $J$ of $S$, the image $\pi_I(J)$ is an ideal
  of $S_I$.
\end{enumerate}
\end{proposition}
\begin{proof}
The first assertion is the definition of the zero-class of $\wp_I$.
The second follows from extensivity $I\subseteq I^g$.
The third is the fixed-point condition $I=I^g$.
Homomorphic images of ideals are ideals.
\end{proof}

Not every ideal of $S$ descends faithfully to the quotient by $\wp_I$.
Saturation is the condition that rules out dependence on the choice
of congruence-class representatives.

\begin{definition}\label{def:I-saturated}
An ideal $J$ of $S$ is \emph{$I$-saturated} if $I^g\subseteq J$ and
$(a,b)\in\wp_I$ together with $a\in J$ imply $b\in J$.
Equivalently, $J=\pi_I^{-1}(\pi_I(J))$.
\end{definition}

The ideal-correspondence theorem for the quotient takes the
following form.
Saturation is what replaces the automatic correspondence theorem
available in rings.

\begin{proposition}\label{prop:quotient-ideal-correspondence}
The assignments $J\mapsto\pi_I(J)$ and $K\mapsto\pi_I^{-1}(K)$
are mutually inverse order-preserving bijections between
$I$-saturated ideals of $S$ and ideals of $S_I$.
\end{proposition}
\begin{proof}
This is the standard ideal correspondence for a quotient by a
congruence, noting that an ideal in the quotient pulls back to an
ideal containing $I^g$ and saturated with respect to $\wp_I$.
\end{proof}

Before transporting maximal-hull definitions to the quotient,
we need to know how maximal congruences behave under the quotient map.

\begin{proposition}\label{prop:quotient-congruence-max}
The maximal congruences of $S_I$ are the congruences
$\theta/\wp_I$ with $\theta\in\MC(I)$.
\end{proposition}
\begin{proof}
Congruences on $S_I$ correspond to congruences on $S$ containing
$\wp_I$.
Since $\wp_I=\bigcap_{\theta\in\MC(I)}\theta$, each $\theta\in\MC(I)$
contains $\wp_I$ and gives a maximal congruence on the quotient.
Conversely, the inverse image of a maximal congruence on the quotient
is a maximal congruence on $S$ containing $\wp_I$;
its zero-class contains $I$, so it belongs to $\MC(I)$.
\end{proof}

The quotient theory of $z$-ideals requires one more controlled
lifting condition, which we isolate as a hypothesis so the subsequent
theorems display their hypotheses explicitly.

\begin{hypothesis}\label{hyp:quotient-lifting}
Let $I$ be a $g$-closed ideal.
We say $I$ satisfies the \emph{maximal-ideal lifting hypothesis}
if every maximal ideal of $S_I$ is of the form $\pi_I(M)$ for a
unique maximal ideal $M\in\mathcal{V}_{\max}(I)$, where
\[
\mathcal{V}_{\max}(I)=\{M\in\Max(S)\mid I\subseteq M\},
\]
and membership is detected by
$[a]_I\in\pi_I(M)\Leftrightarrow a\in M$.
\end{hypothesis}

The quotient remembers only the maximal ideals above $I$,
so the $z$-condition must be relativised accordingly.

\begin{definition}\label{def:I-relative}
An ideal $J$ of $S$ is an \emph{$I$-relative $z$-ideal} if
$\M(a)\cap\mathcal{V}_{\max}(I)=\M(b)\cap\mathcal{V}_{\max}(I)$
and $a\in J$ imply $b\in J$.
\end{definition}

The following theorem shows that, under the lifting condition,
$I$-relative $z$-ideals descend precisely to $z$-ideals in $S_I$.

\begin{theorem}\label{loc11}
Assume $I$ is $g$-closed and satisfies
Hypothesis~\ref{hyp:quotient-lifting}.
Let $J$ be an $I$-saturated ideal.
Then $\pi_I(J)$ is a $z$-ideal of $S_I$ if and only if
$J$ is an $I$-relative $z$-ideal of $S$.
\end{theorem}
\begin{proof}
Suppose first that $\pi_I(J)$ is a $z$-ideal.
Let $a\in J$ and assume
$\M(a)\cap\mathcal{V}_{\max}(I)=\M(b)\cap\mathcal{V}_{\max}(I)$.
Hypothesis~\ref{hyp:quotient-lifting} identifies this with equality
of maximal hulls of $[a]_I$ and $[b]_I$ in $S_I$.
Since $[a]_I\in\pi_I(J)$ and $\pi_I(J)$ is a $z$-ideal,
$[b]_I\in\pi_I(J)$.
Saturation gives $b\in J$.

Conversely, if $J$ is $I$-relative and $[a]_I\in\pi_I(J)$,
choose a representative $a\in J$ using saturation.
If $[a]_I$ and $[b]_I$ have the same maximal hull in $S_I$,
Hypothesis~\ref{hyp:quotient-lifting} gives the relative equality
in $S$; the $I$-relative property then gives $b\in J$ and
$[b]_I\in\pi_I(J)$.
\end{proof}

Intersections in the quotient are controlled by the following lemma,
which is needed for the irreducibility results.

\begin{lemma}\label{loc10}
Let $J,K,L$ be $I$-saturated ideals.  Then
$\pi_I(J)\cap\pi_I(K)=\pi_I(L)$ if and only if $J\cap K=L$.
\end{lemma}
\begin{proof}
Apply the bijection of Proposition~\ref{prop:quotient-ideal-correspondence}
to both sides.
Pulling back $\pi_I(J)\cap\pi_I(K)$ gives $J\cap K$,
and pulling back $\pi_I(L)$ gives $L$.
\end{proof}

Strong irreducibility descends to the quotient under the saturation
assumptions just established.

\begin{proposition}\label{prop:quotient-strong-irred}
Assume $I$ is $g$-closed and satisfies
Hypothesis~\ref{hyp:quotient-lifting}.
Let $J$ be an $I$-saturated $I$-relative $z$-ideal of $S$.
If $J$ is strongly irreducible as an ideal of $S$,
then $\pi_I(J)$ is $z$-strongly irreducible in $S_I$.
\end{proposition}
\begin{proof}
By Theorem~\ref{loc11}, $\pi_I(J)$ is a $z$-ideal.
Let $A,B\in\ZId(S_I)$ with $A\cap B\subseteq\pi_I(J)$.
Put $K=\pi_I^{-1}(A)$ and $L=\pi_I^{-1}(B)$.
Then $K$ and $L$ are $I$-saturated $I$-relative $z$-ideals,
and Lemma~\ref{loc10} gives $K\cap L\subseteq J$.
Strong irreducibility of $J$ gives $K\subseteq J$ or $L\subseteq J$,
hence $A\subseteq\pi_I(J)$ or $B\subseteq\pi_I(J)$.
\end{proof}

The converse direction identifies when $z$-irreducibility in the
quotient reflects back to the source semiring.

\begin{proposition}\label{prop:quotient-irred-reflect}
Assume $I$ is $g$-closed and satisfies
Hypothesis~\ref{hyp:quotient-lifting}.
The ideal $\pi_I(I)$ is $z$-irreducible in $S_I$ if and only if
$I$ is irreducible in the lattice of $I$-saturated $I$-relative
$z$-ideals of $S$.
\end{proposition}
\begin{proof}
If $A\cap B=\pi_I(I)$ for $z$-ideals of $S_I$, the pullbacks
$\pi_I^{-1}(A)$ and $\pi_I^{-1}(B)$ are $I$-saturated $I$-relative
$z$-ideals and Lemma~\ref{loc10} gives their intersection equals $I$.
The converse is the same argument applied in the other direction.
\end{proof}

\section{Under localization}\label{sec:local}

Let $T\subseteq S$ be a multiplicatively closed set containing $1$,
and let $f:S\to T^{-1}S$ be the localization map.
For an ideal $J\subseteq S$ write $T^{-1}J$ for its extension;
for an ideal $K\subseteq T^{-1}S$ write $K^c=f^{-1}(K)$
for its contraction.

Localization preserves the $z$-structure only for that part of the
maximal spectrum that avoids the multiplicative set.
The following hypothesis records the maximal-disjointness condition
needed at each point in the localization arguments.

\begin{hypothesis}\label{hyp:local}
Let $I$ be an ideal of $S$.
The triple $(S,T,I)$ satisfies the \emph{maximal-disjointness
hypothesis} if extension and contraction give a bijection
\[
\mathcal{V}_{\max}(I)\longleftrightarrow\Max(T^{-1}S),
\qquad M\longmapsto T^{-1}M,
\]
and for every $M\in\mathcal{V}_{\max}(I)$ one has $M\cap T=\varnothing$
and $a/s\in T^{-1}M\Leftrightarrow a\in M$ for all $a\in S$,
$s\in T$.
\end{hypothesis}

The following invariance principle shows that the $z$-ideal condition
is preserved by isomorphisms, allowing us to identify isomorphic
copies freely in localization arguments.

\begin{lemma}\label{loc13}
If $\varphi:S\to S'$ is a semiring isomorphism and $J$ is a
$z$-ideal of $S'$, then $\varphi^{-1}(J)$ is a $z$-ideal of $S$.
\end{lemma}
\begin{proof}
An isomorphism induces a bijection between maximal ideals,
so $\M_S(x)=\M_S(y)$ if and only if
$\M_{S'}(\varphi(x))=\M_{S'}(\varphi(y))$.
\end{proof}

The main localization theorem identifies what survives after inverting
$T$; the relative hypothesis ensures the relevant maximal ideals
remain visible in the localized semiring.

\begin{theorem}\label{loc15}
Assume $(S,T,I)$ satisfies Hypothesis~\ref{hyp:local}.
\begin{enumerate}
\item If $J$ is an $I$-relative $z$-ideal of $S$,
  then $T^{-1}J$ is a $z$-ideal of $T^{-1}S$.
\item If $K$ is a $z$-ideal of $T^{-1}S$,
  then $K^c$ is a $z$-ideal of $S$.
\end{enumerate}
\end{theorem}
\begin{proof}
(1) Let $x/s\in T^{-1}J$ and suppose
$\M_{T^{-1}S}(x/s)=\M_{T^{-1}S}(y/t)$.
Choose $u\in T$ with $ux\in J$.
Let $M\in\M(ux)\cap\mathcal{V}_{\max}(I)$.
By Hypothesis~\ref{hyp:local}, $u\notin M$;
since $M$ is prime and $ux\in M$, we get $x\in M$.
Hence $x/s\in T^{-1}M$ and therefore $y/t\in T^{-1}M$.
Hypothesis~\ref{hyp:local} gives $y\in M$.
The reverse inclusion is analogous, giving
$\M(ux)\cap\mathcal{V}_{\max}(I)=\M(y)\cap\mathcal{V}_{\max}(I)$.
Since $J$ is $I$-relative and $ux\in J$, it follows that $y\in J$.

(2) Let $x\in K^c$ and $\M_S(x)=\M_S(y)$.
For any maximal ideal $N=T^{-1}M$ of $T^{-1}S$ containing $x/1$:
Hypothesis~\ref{hyp:local} gives $x\in M$, hence $y\in M$,
and therefore $y/1\in N$.
Thus $\M_{T^{-1}S}(x/1)\subseteq\M_{T^{-1}S}(y/1)$;
the reverse inclusion is symmetric.
Since $K$ is a $z$-ideal and $x/1\in K$, we get $y\in K^c$.
\end{proof}

Extension and contraction encode the passage between source and
target ideals.
The formula below identifies the contraction of an extended ideal
in terms of colon ideals.

\begin{proposition}\label{prop:extension-contraction}
$(T^{-1}J)^c=\bigcup_{t\in T}(J:t)$.
Consequently $T^{-1}J=T^{-1}S$ if and only if $J\cap T\neq\varnothing$.
\end{proposition}
\begin{proof}
$x\in(T^{-1}J)^c\Leftrightarrow x/1\in T^{-1}J\Leftrightarrow
tx\in J$ for some $t\in T$.
\end{proof}

The preceding formula suggests the appropriate saturation notion
for localization: a $T$-saturated ideal is exactly one that is
unchanged after extending and contracting back.

\begin{definition}\label{def:T-saturated}
An ideal $J$ of $S$ is \emph{$T$-saturated} if
$J=\bigcup_{t\in T}(J:t)$,
equivalently if $tx\in J$ for some $t\in T$ implies $x\in J$.
\end{definition}

\begin{corollary}\label{cor:local-contraction-extension}
If $J$ is $T$-saturated, then $(T^{-1}J)^c=J$.
If $K$ is an ideal of $T^{-1}S$, then $T^{-1}(K^c)=K$.
\end{corollary}
\begin{proof}
The first assertion is Proposition~\ref{prop:extension-contraction}.
For the second, every element $a/t\in K$ has $a/1\in K$
after multiplying by the unit $t/1$, so $a\in K^c$.
\end{proof}

Combining saturation with the maximal-disjointness hypothesis
gives the bijective form of localization for relative $z$-ideals.

\begin{proposition}\label{loc19}
Assume Hypothesis~\ref{hyp:local}.  Extension and contraction give
a bijection between $T$-saturated $I$-relative $z$-ideals of $S$
and $z$-ideals of $T^{-1}S$.
Under this bijection, strong irreducibility is preserved and
reflected, testing against the appropriate saturated relative lattice
on the source side.
\end{proposition}
\begin{proof}
The bijection is Corollary~\ref{cor:local-contraction-extension}.
Theorem~\ref{loc15}(1) sends $T$-saturated $I$-relative $z$-ideals
to $z$-ideals after extension.
If $K$ is a $z$-ideal of $T^{-1}S$, then $K^c$ is $T$-saturated;
the proof of Theorem~\ref{loc15}(2), restricted to $\mathcal{V}_{\max}(I)$,
shows $K^c$ is $I$-relative.

For strong irreducibility, suppose $J$ is strongly irreducible in
the source lattice and $A\cap B\subseteq T^{-1}J$ in $\ZId(T^{-1}S)$.
Contracting gives $A^c\cap B^c\subseteq J$;
strong irreducibility yields $A^c\subseteq J$ or $B^c\subseteq J$,
hence $A\subseteq T^{-1}J$ or $B\subseteq T^{-1}J$.
The converse is symmetric.
\end{proof}

Primary ideals test whether localization preserves not only primality
but also radical-controlled factorization.
The following definition provides the $z$-ideal version used here.

\begin{definition}\label{def:zprimary}
A proper $z$-ideal $Q$ is \emph{$z$-primary} if $xy\in Q$ and
$x\notin Q$ imply $y\in\sqrt[z]{Q}$.
\end{definition}

\begin{proposition}\label{loc20}
Assume Hypothesis~\ref{hyp:local}.
Let $Q$ be a proper $T$-saturated ideal that is both an
$I$-relative $z$-ideal and a $z$-primary ideal of $S$,
and assume $\sqrt[z]{Q}\cap T=\varnothing$.
Then $T^{-1}Q$ is a $z$-primary ideal of $T^{-1}S$,
$(T^{-1}Q)^c=Q$, and $T^{-1}K\subseteq T^{-1}Q$ implies
$K\subseteq Q$ for every $z$-ideal $K$ of $S$.
\end{proposition}
\begin{proof}
The ideal $T^{-1}Q$ is a $z$-ideal by Theorem~\ref{loc15}(1).
Suppose $(x/s)(y/t)\in T^{-1}Q$ and $x/s\notin T^{-1}Q$.
Then $uxy\in Q$ for some $u\in T$.
$T$-saturation gives $xy\in Q$, and $x\notin Q$ gives
$y\in\sqrt[z]{Q}$.
The containment $T^{-1}\sqrt[z]{Q}\subseteq\sqrt[z]{T^{-1}Q}$
follows by contracting each $z$-prime ideal of $T^{-1}S$
above $T^{-1}Q$; hence $y/t\in\sqrt[z]{T^{-1}Q}$.
The equality $(T^{-1}Q)^c=Q$ is $T$-saturation.
The final assertion follows since $T^{-1}K\subseteq T^{-1}Q$
and $x\in K$ give $x/1\in(T^{-1}Q)^c=Q$.
\end{proof}

Strong irreducibility of the primary ideal can be transported along
localization.
Since localization sees only the saturated relative lattice,
the reflected statement is phrased accordingly.

\begin{corollary}\label{cor:local-primary-strong}
Under the hypotheses of Proposition~\ref{loc20},
$T^{-1}Q$ is $z$-strongly irreducible if and only if $Q$ is strongly
irreducible in the lattice of $T$-saturated $I$-relative $z$-ideals
of $S$.
In particular, if every $z$-ideal containing $Q$ is $T$-saturated
and $I$-relative, then strong irreducibility of $Q$ and of $T^{-1}Q$
are equivalent.
\end{corollary}
\begin{proof}
This is the strong-irreducibility content of Proposition~\ref{loc19}
applied to $Q$.
Under the additional assumption, testing against all $z$-ideals
containing $Q$ coincides with testing inside the saturated
relative lattice.
\end{proof}

\section{Spectral maps attached to quotients and
localizations}\label{sec:spectral-maps}

The quotient and localization results of the preceding two sections
can be rephrased topologically: the algebraic correspondences
become homeomorphisms between spectral spaces.
This section identifies those homeomorphisms and records their
key properties.
The distinction from the classical ring-theoretic statements is that
the words ``saturated'', ``relative'', and ``disjoint from the
multiplicative set'' encode genuine structural conditions that
are automatic in rings but genuinely additional for semirings.

\subsection{The quotient spectrum}

Keep the notation of Section~\ref{sec:quotient}.
The relative spectrum of the quotient consists of the prime $z$-ideals
of $S$ that are visible from the quotient by $\wp_I$.

\begin{definition}\label{def:relative-z-spectrum}
Assume $I$ is $g$-closed and satisfies
Hypothesis~\ref{hyp:quotient-lifting}.
Let $\Spec_z^I(S)$ be the set of ideals $P$ of $S$ such that:
\begin{enumerate}[label=\textup{(Q\arabic*)}]
\item $P$ is $I$-saturated;
\item $P$ is an $I$-relative $z$-ideal;
\item $P$ is an ordinary prime ideal.
\end{enumerate}
We topologize $\Spec_z^I(S)$ by the closed subbasis
$V_I(A)=\{P\in\Spec_z^I(S):A\subseteq P\}$ for $A\subseteq S$.
\end{definition}

The following lemma gives the prime-ideal correspondence at the
level of points; it is the algebraic core of the quotient homeomorphism.

\begin{lemma}\label{lem:quotient-prime-correspondence}
Assume $I$ is $g$-closed and satisfies
Hypothesis~\ref{hyp:quotient-lifting}.
Let $P$ be an $I$-saturated ideal of $S$.
Then $P$ is prime if and only if $\pi_I(P)$ is prime in $S_I$.
\end{lemma}
\begin{proof}
Suppose $P$ is prime and $[a]_I[b]_I=[p]_I\in\pi_I(P)$ with $p\in P$.
Then $(ab,p)\in\wp_I$; saturation gives $ab\in P$,
hence $a\in P$ or $b\in P$.

Conversely, if $\pi_I(P)$ is prime and $ab\in P$,
then $[a]_I[b]_I=[ab]_I\in\pi_I(P)$,
so $[a]_I\in\pi_I(P)$ or $[b]_I\in\pi_I(P)$,
and saturation gives $a\in P$ or $b\in P$.
\end{proof}

\begin{theorem}\label{thm:quotient-z-spectrum-homeomorphism}
Assume $I$ is $g$-closed and satisfies
Hypothesis~\ref{hyp:quotient-lifting}.
The maps $P\mapsto\pi_I(P)$ and $Q\mapsto\pi_I^{-1}(Q)$ are
mutually inverse homeomorphisms $\Spec_z^I(S)\cong\Spec_z(S_I)$.
\end{theorem}
\begin{proof}
The ideal correspondence of Proposition~\ref{prop:quotient-ideal-correspondence}
gives bijections between $I$-saturated ideals and ideals of $S_I$.
Theorem~\ref{loc11} identifies the $I$-relative $z$-ideals among
the former with the $z$-ideals among the latter.
Lemma~\ref{lem:quotient-prime-correspondence} and
Theorem~\ref{thm:zprime-exchange} identify the prime $z$-ideals.
For the topology, if $B\subseteq S_I$ and $A=\pi_I^{-1}(B)$,
then $\Phi^{-1}(V_z(B))=V_I(A)$;
and the image of $V_I(A)$ is $V_z(\pi_I(A))$.
\end{proof}

\begin{corollary}\label{cor:quotient-z-spectrum-spectral}
Under the hypotheses of Theorem~\ref{thm:quotient-z-spectrum-homeomorphism},
the space $\Spec_z^I(S)$ is spectral.
\end{corollary}
\begin{proof}
It is homeomorphic to $\Spec_z(S_I)$, which is spectral by
Corollary~\ref{cor:zspec-spectral}.
\end{proof}

\begin{remark}\label{rem:quotient-topological-hypotheses}
The theorem is stated for $I$-saturated ideals precisely because
without saturation the implication $[a]_I\in\pi_I(P)\Rightarrow a\in P$
fails, and the proof of Lemma~\ref{lem:quotient-prime-correspondence}
breaks down.
This is the exact point at which the absence of subtraction in
a semiring forces an additional hypothesis.
\end{remark}

\subsection{The localization spectrum}

Keep the notation of Section~\ref{sec:local}, with
Hypothesis~\ref{hyp:local} in force for the chosen ideal $I$.
The visible prime $z$-ideals after localization are those that
avoid $T$ and satisfy the relative $z$-condition.

\begin{definition}\label{def:localized-relative-spectrum}
Let $\Spec_z^{T,I}(S)$ be the set of ideals $P$ of $S$ such that:
\begin{enumerate}[label=\textup{(L\arabic*)}]
\item $P$ is an $I$-relative $z$-ideal;
\item $P$ is an ordinary prime ideal;
\item $P\cap T=\varnothing$.
\end{enumerate}
It is topologized by
$V_{T,I}(A)=\{P\in\Spec_z^{T,I}(S):A\subseteq P\}$ for $A\subseteq S$.
\end{definition}

The following lemma is the point-level localization correspondence
for prime $z$-ideals.

\begin{lemma}\label{lem:local-prime-correspondence}
Assume Hypothesis~\ref{hyp:local}.
If $P\in\Spec_z^{T,I}(S)$, then $T^{-1}P$ is a prime $z$-ideal
of $T^{-1}S$.
If $Q$ is a prime $z$-ideal of $T^{-1}S$, then $Q^c\in\Spec_z^{T,I}(S)$.
\end{lemma}
\begin{proof}
Let $P\in\Spec_z^{T,I}(S)$.
Since $P\cap T=\varnothing$, the extension $T^{-1}P$ is proper.
It is a $z$-ideal by Theorem~\ref{loc15}(1).
If $(a/s)(b/t)\in T^{-1}P$, then $uab\in P$ for some $u\in T$;
primeness of $P$ and $u\notin P$ give $a\in P$ or $b\in P$.

Conversely, $Q^c$ is a $z$-ideal by Theorem~\ref{loc15}(2)
and $I$-relative by Proposition~\ref{loc19}.
It is prime (since $ab\in Q^c$ gives $(a/1)(b/1)\in Q$)
and disjoint from $T$ (since units lie in no proper prime).
\end{proof}

\begin{theorem}\label{thm:local-z-spectrum-homeomorphism}
Assume Hypothesis~\ref{hyp:local}.
Extension and contraction give mutually inverse homeomorphisms
$\Spec_z^{T,I}(S)\cong\Spec_z(T^{-1}S)$.
\end{theorem}
\begin{proof}
For a prime $P$ disjoint from $T$, primeness gives $T$-saturation,
so Corollary~\ref{cor:local-contraction-extension} gives
$(T^{-1}P)^c=P$.
Lemma~\ref{lem:local-prime-correspondence} restricts the
extension-contraction bijections to the stated sets of $z$-prime ideals.
For closed sets, $B\subseteq T^{-1}P$ is equivalent to
$f^{-1}(\langle B\rangle)\subseteq P$, giving
$\{P:B\subseteq T^{-1}P\}=V_{T,I}(f^{-1}(\langle B\rangle))$;
and the image of $V_{T,I}(A)$ is $V_z(T^{-1}\langle A\rangle)$.
\end{proof}

\begin{corollary}\label{cor:local-relative-spectrum-spectral}
Under Hypothesis~\ref{hyp:local}, $\Spec_z^{T,I}(S)$ is spectral.
\end{corollary}
\begin{proof}
By Theorem~\ref{thm:local-z-spectrum-homeomorphism} it is homeomorphic
to $\Spec_z(T^{-1}S)$, which is spectral by Corollary~\ref{cor:zspec-spectral}.
\end{proof}

\begin{remark}\label{rem:local-topological-hypotheses}
The disjointness condition $P\cap T=\varnothing$ is not optional:
if $P$ meets $T$, then $T^{-1}P=T^{-1}S$ and no prime is obtained.
Similarly, the $I$-relative condition is what makes
Theorem~\ref{loc15}(1) applicable; an arbitrary $z$-ideal of $S$
need not have $z$-closed extension after localization if maximal
ideals above relevant elements disappear.
\end{remark}

\section{The frame of \texorpdfstring{$z$}{z}-ideals}
\label{sec:frame}

For $I,J\in\ZId(S)$ define
$I\vee_z J=\clz(I+J)$ and $I\wedge J=I\cap J$,
extending joins to arbitrary families by applying $\clz$ to the sum.
The bottom element is $\clz(0)$ and the top element is $S$.

\subsection{The coherent frame of
\texorpdfstring{$z$}{z}-ideals}

The following theorem collects the finite-type properties of the
ordinary $z$-closure from~\cite{AG23,AG25}.
These are precisely the inputs needed in the coherent-frame argument,
and they are gathered here to make the logical dependencies explicit.

\begin{theorem}[Finite-type $z$-frame calculus,
\cite{AG23,AG25}]\label{lem:z-finite-type}
For every commutative semiring $S$:
\begin{enumerate}
\item $\clz$ is a finite-type closure operation on ideals;
\item with meet $I\cap J$ and joins
  $\bigvee^z I_\lambda=\clz(\sum_\lambda I_\lambda)$,
  the ordered set $\ZId(S)$ is a frame;
\item every $z$-ideal is the join of the principal $z$-closures
  $\ideal{a}_z$ with $a\in I$, and the compact elements are
  precisely the finite joins of such closures;
\item for all $a,b\in S$,
  $\ideal{a}_z\cap\ideal{b}_z=\ideal{ab}_z$,
  and compact elements are closed under finite meets.
\end{enumerate}
\end{theorem}
\begin{proof}
The finite-type property and the principal-generation statement
are the finite-type $z$-closure theorem of~\cite{AG23}.
The exchange theorem and product formula of~\cite{AG25}
give the meet identity.
The frame law and compact-element description follow from
finite type and the product formula.
\end{proof}

Before establishing coherence, we first confirm the frame structure
on which the compactness argument rests.

\begin{theorem}\label{loc22}
For every commutative semiring $S$, $\ZId(S)$ is a frame.
\end{theorem}
\begin{proof}
Theorem~\ref{lem:z-finite-type}(2).
\end{proof}

The compact elements of the frame are the finitely generated
$z$-ideals; the following lemma makes this identification explicit.

\begin{lemma}\label{loc25}
The compact elements of $\ZId(S)$ are the finite joins
$\ideal{a_1}_z\vee_z\cdots\vee_z\ideal{a_n}_z$.
\end{lemma}
\begin{proof}
Theorem~\ref{lem:z-finite-type}(3).
\end{proof}

Coherence requires compact elements to be closed under finite meets.
The product formula below provides the maximal-hull reason this holds.

\begin{lemma}\label{loc23}
For all $a,b\in S$, $\ideal{a}_z\cap\ideal{b}_z=\ideal{ab}_z$.
Consequently, compact elements of $\ZId(S)$ are closed under finite
meets.
\end{lemma}
\begin{proof}
Theorem~\ref{lem:z-finite-type}(4).
\end{proof}

We now combine the frame structure, compact generation,
and the finite-meet formula to prove Theorem~B.

\begin{theorem}\label{ZId(S)-is-coherent}
For every commutative semiring $S$, $\ZId(S)$ is a coherent frame.
\end{theorem}
\begin{proof}
The top element $S=\ideal{1}_z$ is compact.
By Lemma~\ref{loc25}, every element is a join of compact elements.
By Lemma~\ref{loc23}, compact elements are closed under finite meets.
Hence $\ZId(S)$ is coherent.
\end{proof}

The passage from a coherent frame to a spectral prime spectrum is
standard Stone duality.
The corollary below spells out the topological consequences,
including the explicit form of the basic quasi-compact open sets.

\begin{corollary}\label{cor:zspec-spectral}
The space $\Spec_z(S)$ of prime $z$-ideals is spectral.
The basic opens
\[
D_z(a)=\{P\in\Spec_z(S)\mid a\notin P\}
\]
are quasi-compact and satisfy $D_z(a)\cap D_z(b)=D_z(ab)$.
\end{corollary}
\begin{proof}
Prime elements of the coherent frame $\ZId(S)$ are the $z$-prime
ideals by Corollary~\ref{cor:ordinary-prime}.
The spectrum of a coherent frame is spectral.
The formula for basic opens is dual to $\ideal{a}_z\cap\ideal{b}_z
=\ideal{ab}_z$.
\end{proof}

\subsection{The coherent frame of \texorpdfstring{$g$}{g}-closed
ideals}

The analogous theorem for $g$-closed ideals depends on hypotheses
that are automatic in the classical ring setting but genuinely
additional for semirings.
Example~\ref{ex:zero-class-distinction} is one reason the
congruence-generated closure must be kept separate from the ordinary
maximal-ideal $z$-closure;
we therefore state the required conditions as an explicit hypothesis.

\begin{hypothesis}\label{hyp:g-closure}
A semiring $S$ satisfies the \emph{finite-type $g$-closure
hypothesis} if:
\begin{enumerate}[label=\upshape(G\arabic*)]
\item the closure $g:\Id(S)\to\Id(S)$ is of finite type;
\item the fixed points of $g$ form a frame with meet given by
  intersection and join given by $g$ applied to sums;
\item the compact elements of this fixed-point frame are exactly
  the finite $g$-joins of principal $g$-closures
  $\ideal{a_1}_g\vee_g\cdots\vee_g\ideal{a_n}_g$,
  where $\ideal{a}_g=(\ideal{a})^g$;
\item principal compact generators satisfy
  $\ideal{a}_g\cap\ideal{b}_g=\ideal{ab}_g$ for all $a,b\in S$,
  and hence finite $g$-joins of principal $g$-closures are closed
  under finite meets.
\end{enumerate}
\end{hypothesis}

\begin{remark}
Hypothesis~\ref{hyp:g-closure} is automatic in several standard
classes---for example, when the maximal-congruence hull construction
is a finite-type nucleus on the relevant radical-ideal frame.
It is not part of the definition of the canonical $g$-closure itself.
Example~\ref{ex:zero-class-distinction} shows why the
congruence-generated closure must be kept separate from the
ordinary maximal-ideal $z$-closure.
\end{remark}

Under Hypothesis~\ref{hyp:g-closure}, the compact-generation
argument of Theorem~B can be repeated for $g$-closed ideals,
giving the congruence-generated counterpart.

\begin{theorem}\label{GCl-coherent-frame}
If $S$ satisfies Hypothesis~\ref{hyp:g-closure}, then $\GCl(S)$
is a coherent frame.
Its compact elements are the finite $g$-joins
$\ideal{a_1}_g\vee_g\cdots\vee_g\ideal{a_n}_g$.
\end{theorem}
\begin{proof}
Hypothesis (G2) gives the frame structure.
Hypothesis (G3) identifies the compact elements,
and (G4) says they are closed under finite meets.
The top element $S=\ideal{1}_g$ is compact.
\end{proof}

Prime points of the $g$-closed frame are tested against products
of $g$-closed ideals.

\begin{definition}\label{def:g-prime}
A proper $g$-closed ideal $P$ is \emph{$g$-prime} if
$AB\subseteq P$ implies $A\subseteq P$ or $B\subseteq P$
for all $A,B\in\GCl(S)$.
The set of such ideals is $\Spec_g(S)$.
\end{definition}

The following lemma is the $g$-closed analogue of the exchange
principle for $z$-prime ideals.
It shows that frame-primality, $g$-primality, and ordinary
prime-ideal primality are equivalent under
Hypothesis~\ref{hyp:g-closure}.

\begin{lemma}\label{exchange-principal-g-prime}
Assume Hypothesis~\ref{hyp:g-closure}.
For a proper $g$-closed ideal $P$, the following are equivalent:
\begin{enumerate}[label=\upshape(\roman*)]
\item $P$ is a prime element of the frame $\GCl(S)$;
\item $P$ is $g$-prime;
\item $P$ is an ordinary prime ideal of $S$.
\end{enumerate}
\end{lemma}
\begin{proof}
Assume $P$ is a prime element of $\GCl(S)$ and $ab\in P$.
Since $P$ is $g$-closed, $\ideal{ab}_g\subseteq P$.
Hypothesis (G4) gives
$\ideal{a}_g\cap\ideal{b}_g=\ideal{ab}_g\subseteq P$;
frame-primality gives $\ideal{a}_g\subseteq P$ or
$\ideal{b}_g\subseteq P$, hence $a\in P$ or $b\in P$.

If $P$ is ordinary prime and $AB\subseteq P$ with $A,B\in\GCl(S)$,
choose $a\in A\setminus P$ and $b\in B\setminus P$ to get
$ab\notin P$, contradicting $ab\in AB\subseteq P$.
Hence $A\subseteq P$ or $B\subseteq P$, and $P$ is $g$-prime.

If $P$ is $g$-prime and $A\cap B\subseteq P$, then
$AB\subseteq A\cap B\subseteq P$, so $A\subseteq P$ or $B\subseteq P$.
This is frame-primality.
\end{proof}

For $A\subseteq S$, define $V_G(A)=\{P\in\Spec_g(S)\mid A\subseteq P\}$.
Closed sets in the $g$-spectrum satisfy the identities one expects
from an ideal-theoretic hull-kernel topology.

\begin{lemma}\label{loc31}
Assume Hypothesis~\ref{hyp:g-closure}.  The sets $V_G(A)$ satisfy:
\begin{enumerate}
\item $V_G(0)=\Spec_g(S)$ and $V_G(S)=\varnothing$;
\item $A\subseteq B$ implies $V_G(B)\subseteq V_G(A)$;
\item for ideals $I,J$,
  $V_G(I)\cup V_G(J)=V_G(IJ)=V_G(I\cap J)$;
\item $\bigcap_\lambda V_G(J_\lambda)=V_G(\sum_\lambda J_\lambda)$.
\end{enumerate}
\end{lemma}
\begin{proof}
Parts (1), (2), and (4) are formal.
Part (3) follows from Lemma~\ref{exchange-principal-g-prime}:
a prime ideal contains $IJ$ if and only if it contains $I$ or $J$,
and $IJ\subseteq I\cap J$ gives the equality with $V_G(I\cap J)$.
\end{proof}

The spectrality of $\Spec_g(S)$ now follows from the
coherent-frame theorem via Stone duality.

\begin{theorem}\label{Spec_g-is-spectral}
If $S$ satisfies Hypothesis~\ref{hyp:g-closure}, then $\Spec_g(S)$,
with the topology whose closed sets are the $V_G(A)$, is spectral.
\end{theorem}
\begin{proof}
Theorem~\ref{GCl-coherent-frame} gives a coherent frame $\GCl(S)$.
Lemma~\ref{exchange-principal-g-prime} identifies its prime elements
with the $g$-prime ideals, and the closed sets $V_G(A)$ are the
hull-kernel closed sets.
The spectrum of a coherent frame is spectral.
\end{proof}

\subsection{Canonical \texorpdfstring{$g$}{g}-congruence spectrum}

A canonical $g$-congruence $\rho$ is \emph{prime} if
$(ab,0)\in\rho$ implies $(a,0)\in\rho$ or $(b,0)\in\rho$.
Let $\GCSpec(S)$ denote the set of canonical prime $g$-congruences.
For $A\subseteq S$ define
$W_G(A)=\{\rho\in\GCSpec(S)\mid A\subseteq\zero\rho\}$.

To pass from ideals to congruences, we need the prime objects
to correspond under the zero-class and associated-congruence maps.

\begin{lemma}\label{loc30}
Assume Hypothesis~\ref{hyp:g-closure}.  The assignments
$P\mapsto\wp_P$ and $\rho\mapsto\zero\rho$ are mutually inverse
bijections between $\Spec_g(S)$ and $\GCSpec(S)$.
\end{lemma}
\begin{proof}
The assignments are mutually inverse on $g$-closed ideals and
canonical $g$-congruences by Proposition~\ref{prop:g-correspondence-basic}.
If $P$ is $g$-prime and $(ab,0)\in\wp_P$, then $ab\in P$,
so $a\in P$ or $b\in P$, giving $(a,0)\in\wp_P$ or $(b,0)\in\wp_P$.
Conversely, if $\rho$ is prime, then $\zero\rho$ is an ordinary prime
ideal, hence $g$-prime by Lemma~\ref{exchange-principal-g-prime}.
\end{proof}

The closed-set calculus for canonical prime $g$-congruences mirrors
that for $g$-closed ideals.

\begin{lemma}\label{loc38}
Assume Hypothesis~\ref{hyp:g-closure}.  The sets $W_G(A)$ satisfy
the same closed-set identities as the sets $V_G(A)$ in
Lemma~\ref{loc31}.
\end{lemma}
\begin{proof}
Transport Lemma~\ref{loc31} through the bijection
$\rho\mapsto\zero\rho$ of Lemma~\ref{loc30}.
\end{proof}

The following theorem gives the homeomorphism between the ideal
spectrum and the congruence spectrum, completing the proof of
Theorem~C.

\begin{theorem}\label{loc34}
Assume Hypothesis~\ref{hyp:g-closure}.  The spaces $\Spec_g(S)$
and $\GCSpec(S)$ are homeomorphic via $P\mapsto\wp_P$ and
$\rho\mapsto\zero\rho$.
\end{theorem}
\begin{proof}
Lemma~\ref{loc30} gives inverse bijections $z(P)=\wp_P$ and
$c(\rho)=\zero\rho$.  Moreover,
$z^{-1}(W_G(A))=V_G(A)$ and $c^{-1}(V_G(A))=W_G(A)$,
so both maps are continuous.
\end{proof}

\begin{corollary}\label{cor:gcong-spectral}
If $S$ satisfies Hypothesis~\ref{hyp:g-closure},
then $\GCSpec(S)$ is spectral.
\end{corollary}
\begin{proof}
Combine Theorem~\ref{Spec_g-is-spectral} and Theorem~\ref{loc34}.
\end{proof}

\section{Model examples and necessity of the
hypotheses}\label{sec:examples}

The preceding sections separate unconditional results from those
requiring compatibility between maximal ideals and maximal congruences.
The computations in this section make that separation explicit.
In particular, they show that the three closures---ordinary $z$-closure,
$z_k$-closure, and $g$-closure---are genuinely distinct operations.

\subsection{The natural-number semiring}

We continue to write $k_d$ for congruence modulo $d$ on $\N$
and $\beta=(\N_{>0}\times\N_{>0})\cup\{(0,0)\}$
for the Rees congruence.
The maximal congruences of $\N$ are $\beta$ and the $k_p$ with $p$
prime; their quotients are the Boolean semifield and $\mathds F_p$
respectively.

The following proposition makes explicit which modular congruences
are $z$-congruences.

\begin{proposition}\label{prop:N-zcong-squarefree}
For $d\geqslant2$, the congruence $k_d$ on $\N$ is a $z$-congruence
if and only if $d$ is square-free.
\end{proposition}
\begin{proof}
The maximal congruences containing $(a,b)\in\N\times\N$ are:
$\beta$ when $a,b$ are both positive or both zero,
and $k_p$ when $p\mid a-b$.
If $d$ is square-free and $(a,b)\in k_d$, every prime $p\mid d$
gives $k_p\in\MC(a,b)$;
if $\MC(a,b)\subseteq\MC(c,e)$, every such $p$ divides $c-e$,
and square-freeness gives $d\mid c-e$.

If $d$ is not square-free, let $r=\prod_{p\mid d}p$.
Then $(0,d)\in k_d$ and
$\MC(0,d)=\{k_p:p\mid d\}=\MC(0,r)$,
but $(0,r)\notin k_d$ because $d\nmid r$.
\end{proof}

The canonical $g$-closure on principal ideals of $\N$ has an
explicit arithmetic description: it replaces the generating integer
by its square-free radical.

\begin{proposition}\label{prop:N-gclosure-principal}
For $n\geqslant2$, $(n\N)^g=\mathrm{rad}(n)\N$,
where $\mathrm{rad}(n)$ is the product of the distinct prime divisors
of $n$.
Moreover $\{0\}^g=\{0\}$, while $(\N\setminus\{1\})^g=\N$.
\end{proposition}
\begin{proof}
The zero-class of $\beta$ is $\{0\}$, so $\beta\notin\MC(n\N)$.
For a prime $p$, $\zero{k_p}=p\N$ and $n\N\subseteq p\N
\Leftrightarrow p\mid n$; hence $\MC(n\N)=\{k_p:p\mid n\}$.
The intersection of these congruences is congruence modulo
$\mathrm{rad}(n)$, with zero-class $\mathrm{rad}(n)\N$.

For $\{0\}$: the presence of $\beta$ in $\MC(\{0\})$ forces every
positive integer out of the zero-class, giving $\{0\}^g=\{0\}$.
For $\N\setminus\{1\}$: no maximal congruence has zero-class
containing this ideal (zero-classes are $\{0\}$ and $p\N$),
so $\MC(\N\setminus\{1\})=\varnothing$, whence
$\wp_{\N\setminus\{1\}}=\N\times\N$ and the $g$-closure is $\N$.
\end{proof}

This arithmetic calculation illustrates the separation between
ordinary $z$-ideals and $g$-closed ideals in a transparent setting.

\begin{corollary}\label{cor:N-gclosed-many}
The semiring $\N$ has many proper $g$-closed ideals that are not
ordinary $z$-ideals.
For instance, $6\N$ is $g$-closed (since $6$ is square-free)
but is not a $z$-ideal.
\end{corollary}
\begin{proof}
By Proposition~\ref{prop:N-gclosure-principal},
$n\N$ is $g$-closed whenever $n$ is square-free.
Example~\ref{ex:N-zideals} shows that the only proper ordinary
$z$-ideal of $\N$ is $\N\setminus\{1\}$.
\end{proof}

The same square-free condition governs $z_k$-ideals, confirming
that the $k$-maximal perspective differs from both the ordinary
and congruence-theoretic ones.

\begin{proposition}\label{prop:N-zk-principal}
For $n\geqslant2$, the ideal $n\N$ is a $z_k$-ideal of $\N$ if and
only if $n$ is square-free.  The ideal $\{0\}$ is also a $z_k$-ideal.
\end{proposition}
\begin{proof}
The maximal $k$-ideals of $\N$ are $p\N$ for primes $p$, so
$\Mk(a)=\{p\N:p\mid a\}$ for $a>0$ and $\Mk(0)$ comprises
all maximal $k$-ideals.
If $n$ is square-free and $a\in n\N$ with $\Mk(a)=\Mk(b)$,
the same primes divide $b$, so $b\in n\N$.
If $n$ is not square-free and $r=\mathrm{rad}(n)$,
then $n\in n\N$ with $\Mk(n)=\Mk(r)$ but $r\notin n\N$.
Finally, $\Mk(0)=\Mk(b)$ forces every prime to divide $b$,
giving $b=0$; hence $\{0\}$ is $z_k$-closed.
\end{proof}

These computations confirm that the ordinary $z$-closure,
the $z_k$-closure, and the $g$-closure are three genuinely distinct
operations in $\N$:
the ordinary maximal-ideal hull sees only the distinction between
units and non-units; the $k$-maximal hull records prime divisors
through subtractive ideals; the $g$-closure records prime divisors
through congruences.

\subsection{Power-set semirings and Boolean behaviour}

Let $X$ be a set and let $\mathcal{P}(X)$ carry its Boolean semiring
structure $(\cup,\cap,\varnothing,X)$.
For an ideal $\mathcal{I}$ of the Boolean algebra $\mathcal{P}(X)$,
define $(A\equiv_{\mathcal{I}} B)\Leftrightarrow(A\triangle B\in\mathcal{I})$,
where $\triangle$ denotes symmetric difference.

Power-set semirings provide the contrasting Boolean model,
one in which congruences, ideals, and maximal-hull arguments align
particularly cleanly.
The next lemma makes the alignment explicit.

\begin{lemma}\label{lem:boolean-congruence-ideals}
The assignment $\mathcal{I}\mapsto{\equiv_{\mathcal{I}}}$ is an
inclusion-preserving bijection between ideals of the Boolean algebra
$\mathcal{P}(X)$ and semiring congruences on $\mathcal{P}(X)$.
The zero-class of $\equiv_{\mathcal{I}}$ is $\mathcal{I}$.
\end{lemma}
\begin{proof}
If $A\triangle B\in\mathcal{I}$ and $C\triangle D\in\mathcal{I}$,
then $(A\cup C)\triangle(B\cup D)\subseteq(A\triangle B)\cup(C\triangle D)$
and $(A\cap C)\triangle(B\cap D)\subseteq(A\triangle B)\cup(C\triangle D)$;
since $\mathcal{I}$ is closed under finite unions and downward closed,
$\equiv_{\mathcal{I}}$ is a semiring congruence.
Its zero-class is $\{A:A\triangle\varnothing\in\mathcal{I}\}=\mathcal{I}$.

Conversely, if $\theta$ is a congruence with zero-class $\mathcal{I}$,
multiply any pair $(A,B)\in\theta$ by $X\setminus B$ to get
$(A\setminus B,\varnothing)\in\theta$, so $A\setminus B\in\mathcal{I}$;
symmetrically $B\setminus A\in\mathcal{I}$, so $A\triangle B\in\mathcal{I}$.
The reverse direction shows $\theta=\equiv_{\mathcal{I}}$.
\end{proof}

In the Boolean setting, every congruence is already a $z$-congruence.
This proposition explains why Boolean semirings are a particularly
simple class to which Theorem~C applies unconditionally.

\begin{proposition}\label{prop:boolean-all-zcong}
Every congruence on $\mathcal{P}(X)$ is a $z$-congruence.
Moreover every ideal of $\mathcal{P}(X)$ is $g$-closed.
\end{proposition}
\begin{proof}
Maximal congruences correspond by Lemma~\ref{lem:boolean-congruence-ideals}
to maximal ideals of the Boolean algebra, equivalently to ultrafilters.
By the Boolean prime ideal theorem, every Boolean ideal is the
intersection of all maximal ideals containing it.

Let $\theta=\equiv_{\mathcal{I}}$ and suppose $(A,B)\in\theta$ with
$\MC(A,B)\subseteq\MC(C,D)$.
Then $A\triangle B\in\mathcal{I}$.
Every maximal ideal containing $\mathcal{I}$ contains $A\triangle B$,
and the hull inclusion then forces $C\triangle D$ into that same
maximal ideal.
Since $\mathcal{I}$ is the intersection of the maximal ideals
containing it, $C\triangle D\in\mathcal{I}$, so $(C,D)\in\theta$.

For $g$-closedness: $\wp_{\mathcal{I}}=\equiv_{\mathcal{I}}$
by the same Boolean argument, and $\zero{\equiv_{\mathcal{I}}}=\mathcal{I}$.
\end{proof}

\begin{corollary}\label{cor:boolean-g-hypothesis}
For every set $X$, the semiring $\mathcal{P}(X)$ satisfies
Hypothesis~\ref{hyp:g-closure}.
The frame $\GCl(\mathcal{P}(X))$ is the ordinary ideal frame of
the Boolean algebra $\mathcal{P}(X)$.
\end{corollary}
\begin{proof}
By Proposition~\ref{prop:boolean-all-zcong}, $g$ is the identity on
ideals.
The ideal lattice of a Boolean algebra is algebraic and distributive;
principal ideals $\downarrow A=\{B:B\subseteq A\}$ are compact,
$\downarrow A\cap\downarrow B=\downarrow(A\cap B)$, and every ideal
is a directed union of principal ideals.
Hence the identity closure satisfies all parts of
Hypothesis~\ref{hyp:g-closure}.
\end{proof}

\subsection{Two boundary examples}

The following examples delimit the terminology in
Theorem~\ref{thm:regularity} and the hypotheses in the
congruence spectrum theorem.

The first prevents a terminological misunderstanding about the notion
of strong $z$-ideal used in Theorem~A.

\begin{example}\label{ex:additive-strongness}
The phrase ``strong $z$-ideal'' in Theorem~\ref{thm:regularity}
cannot mean ``strong ideal'' in the additive sense.
Let $S=\mathds{F}_2$, which is von Neumann regular with
$E(S)=\mathrm{Comp}(S)=\{0,1\}$.
The zero ideal is not additively strong:
$1+1=0\in\{0\}$ while $1\notin\{0\}$.
Nevertheless it is a strong $z$-ideal in the sense of
Definition~\ref{def:strong-z}, since $\{0\}=\bigcap\{M:\{0\}\subseteq M\}$
trivially.
The additive strong-ideal condition is thus strictly stronger than
the maximal-hull condition, and the two notions must not be
conflated.
\end{example}

The second example is the basic congruence-theoretic test case.
It shows concretely that maximal congruences do not automatically
have maximal zero-classes.

\begin{example}\label{ex:max-cong-zero-class-not-max}
The zero-class of a maximal congruence need not be maximal as an
ideal.
In $\N$, the maximal congruence $\beta$ has zero-class $\{0\}$,
which is properly contained in $p\N$ for every prime $p$.
By Example~\ref{ex:N-zideals}, $\{0\}$ is not an ordinary $z$-ideal.
Hence one cannot pass from maximal congruences to maximal ideals
by taking zero-classes without additional hypotheses.
\end{example}

\begin{remark}\label{rem:example-conclusion}
The examples in this section reveal three structurally independent
layers.
Ordinary $z$-ideals are controlled by ordinary maximal ideals.
The $z_k$-variant is controlled by maximal subtractive ideals.
The $g$-closed ideals are controlled by maximal congruences.
These three controls coincide in ring-like contexts, but they
separate cleanly in basic semirings such as $\N$.
The hypotheses appearing in Sections~\ref{sec:quotient},
\ref{sec:local}, and~\ref{sec:frame} are precisely the conditions
needed to move information between these three layers.
\end{remark}

\section{Two functors}\label{sec:functors}

The coherent-frame theorems of Section~\ref{sec:frame} have a
natural functorial content.
This section makes that content precise,
constructing the two coherent-frame-valued functors $\ZId$ and
$\GCl$ and the natural transformation that connects them.

Let $\mathsf{CRig}_{z}$ be the category whose objects are commutative
semirings and whose morphisms $\varphi:S\to T$ satisfy the
contraction condition: for every $z$-ideal $J$ of $T$,
the inverse image $\varphi^{-1}(J)$ is a $z$-ideal of $S$.

Functoriality is controlled by compact generators.
The following lemma shows that the formula on finite joins of
principal $z$-closures is well-defined and preserves the
lattice operations.

\begin{lemma}\label{lem:compact-map-z}
Let $\varphi:S\to T$ be a morphism in $\mathsf{CRig}_{z}$.  The formula
\[
\overline\varphi\!\left(\ideal{a_1}_z\vee_z\cdots\vee_z\ideal{a_n}_z\right)
=\ideal{\varphi(a_1)}_z\vee_z\cdots\vee_z\ideal{\varphi(a_n)}_z
\]
defines a homomorphism from the compact elements of $\ZId(S)$ to
those of $\ZId(T)$.
\end{lemma}
\begin{proof}
For well-definedness, suppose a finite join generated by the $a_i$
is contained in one generated by the $b_j$.
Let $J$ be the corresponding join of $\ideal{\varphi(b_j)}_z$ in $T$.
Since $\varphi^{-1}(J)$ is a $z$-ideal containing each $b_j$,
it contains the left-hand join, so $\varphi(a_i)\in J$ for every $i$.
Preservation of finite joins is immediate;
finite meets follow from $\ideal{a}_z\cap\ideal{b}_z=\ideal{ab}_z$
and $\varphi(ab)=\varphi(a)\varphi(b)$.
\end{proof}

\begin{proposition}\label{loc35}
The assignment $\ZId(\varphi)(I)=\bigvee_{a\in I}^{z}
\ideal{\varphi(a)}_z$ defines a coherent frame homomorphism
$\ZId(S)\to\ZId(T)$.  Thus
\[
\ZId:\mathsf{CRig}_{z}\longrightarrow\mathsf{CohFrm}
\]
is a functor.
\end{proposition}
\begin{proof}
A coherent frame homomorphism from an algebraic coherent frame is
determined by its action on compact generators, provided that action
preserves finite meets and joins.
Lemma~\ref{lem:compact-map-z} supplies this.
Functoriality is checked on compact generators $\ideal{a}_z$:
$\ZId(\psi\circ\varphi)(\ideal{a}_z)=\ideal{\psi(\varphi(a))}_z
=(\ZId(\psi)\circ\ZId(\varphi))(\ideal{a}_z)$.
\end{proof}

Let $\mathsf{CRig}_{g}$ be the subcategory of $\mathsf{CRig}_{z}$
whose objects satisfy Hypothesis~\ref{hyp:g-closure}
and whose morphisms also satisfy:
inverse images of $g$-closed ideals of $T$ are $g$-closed in $S$.

The $g$-closure functor has the same formal shape but requires the
contraction and $g$-closure hypothesis on both source and target.

\begin{proposition}\label{loc36}
For $\varphi:S\to T$ in $\mathsf{CRig}_{g}$, the formula
$\GCl(\varphi)(I)=\bigvee_{a\in I}^{g}\ideal{\varphi(a)}_g$
defines a coherent frame homomorphism $\GCl(S)\to\GCl(T)$.  Hence
\[
\GCl:\mathsf{CRig}_{g}\longrightarrow\mathsf{CohFrm}
\]
is a functor.
\end{proposition}
\begin{proof}
Both frames are compactly generated by finite $g$-joins of principal
compact elements.
Define the map on generators by $\ideal{a}_g\mapsto\ideal{\varphi(a)}_g$.
Compatibility with finite meets:
$\ideal{a}_g\cap\ideal{b}_g=\ideal{ab}_g$ by Hypothesis (G4),
and $\ideal{\varphi(a)}_g\cap\ideal{\varphi(b)}_g
=\ideal{\varphi(ab)}_g$ by the same hypothesis in $T$.
The inverse-image condition gives the well-definedness check:
containment relations among finite $g$-joins in $S$ are detected
after pulling back the corresponding $g$-closed ideal of $T$.
The displayed formula is the unique extension to all $g$-closed ideals.
\end{proof}

Let $\mathsf{CRig}_{g}^{\sigma}$ be the subcategory of
$\mathsf{CRig}_{g}$ whose objects $S$ have the property that
\[
\sigma_S:\ZId(S)\longrightarrow\GCl(S),\qquad\sigma_S(I)=I^g,
\]
is a coherent frame homomorphism, and whose morphisms $\varphi:S\to T$
satisfy the generator compatibility condition
$\GCl(\varphi)((\ideal{a}_z)^g)=(\ideal{\varphi(a)}_z)^g$
for every $a\in S$.

The canonical closure from $z$-ideals to $g$-closed ideals
assembles into a natural transformation under this compatibility condition.

\begin{proposition}\label{prop:natural-transformation}
On $\mathsf{CRig}_{g}^{\sigma}$, the maps $\sigma_S$ form a
natural transformation $\sigma:\ZId\Rightarrow\GCl$.
\end{proposition}
\begin{proof}
The object condition ensures $\sigma_S\in\mathsf{CohFrm}$.
For a compact generator $\ideal{a}_z$,
the morphism condition gives
\[
\GCl(\varphi)(\sigma_S(\ideal{a}_z))
=\GCl(\varphi)((\ideal{a}_z)^g)
=(\ideal{\varphi(a)}_z)^g
=\sigma_T(\ideal{\varphi(a)}_z)
=\sigma_T(\ZId(\varphi)(\ideal{a}_z)).
\]
Since compact generators join-generate $\ZId(S)$ and all maps
preserve joins, the naturality square commutes everywhere.
\end{proof}

\begin{remark}
The restriction to $\mathsf{CRig}_{g}^{\sigma}$ is essential:
without the frame-homomorphism and generator-compatibility conditions,
the closure map $I\mapsto I^g$ need not define a natural
transformation from the $z$-ideal functor to the $g$-closed ideal
functor.
\end{remark}

\appendix

\section{Auxiliary structural lemmas}\label{app:auxiliary}

This appendix collects technical verifications used in the body of
the paper.
They concern principal $z$-closures, the natural-number semiring,
power-set semirings, associated congruences, quotients,
localizations, coherence, and functoriality.
They are gathered here so that the main text can proceed without
repeating routine verifications, while all hypotheses in the
semiring arguments remain explicit.

\subsection{Maximal hulls and principal
\texorpdfstring{$z$}{z}-closures}

For $a\in S$, recall $\M(a)=\{M\in\Max(S):a\in M\}$
and $\mfrak(a)=\bigcap_{M\in\M(a)}M$ (with $\mfrak(a)=S$
when $a$ is a unit).

\begin{lemma}\label{app:maximal-hull-calculus}
For all $a,b\in S$:
\begin{enumerate}
\item $\M(ab)=\M(a)\cup\M(b)$;
\item $\M(a)\cap\M(b)\subseteq\M(a+b)$;
\item $\ideal{a}_z=\mfrak(a)$.
\end{enumerate}
\end{lemma}
\begin{proof}
Part (1) is Lemma~\ref{lem:max-prime}.
Part (2) follows from closure of ideals under addition.

For (3), $\mfrak(a)$ is a $z$-ideal: if $x\in\mfrak(a)$ and
$\M(x)=\M(y)$, then every maximal ideal containing $a$ contains $x$,
hence $y$, so $y\in\mfrak(a)$.
Thus $\ideal{a}_z\subseteq\mfrak(a)$.
Conversely, let $I$ be any $z$-ideal containing $a$ and
$x\in\mfrak(a)$.
Put $c=ax$; then $c\in I$ and $\M(c)=\M(a)\cup\M(x)=\M(x)$
(using (1) and $\M(a)\subseteq\M(x)$),
so the $z$-ideal property gives $x\in I$.
Intersecting over all such $I$ gives $\mfrak(a)\subseteq\ideal{a}_z$.
\end{proof}

\begin{lemma}\label{app:idempotent-sum-expanded}
Assume $S$ is additively idempotent.  Every $z$-ideal is a $k$-ideal
if and only if $\M(a+b)=\M(a)\cap\M(b)$ for all $a,b\in S$.
\end{lemma}
\begin{proof}
Assume every $z$-ideal is a $k$-ideal.
Each maximal ideal is a $z$-ideal; hence a $k$-ideal.
If $a+b\in M$, idempotence gives $a,b\leqslant a+b$;
subtractivity gives $a,b\in M$.
Hence $\M(a+b)\subseteq\M(a)\cap\M(b)$;
the reverse inclusion is Lemma~\ref{app:maximal-hull-calculus}(2).

Conversely, assume the hull equality.
Let $I$ be a $z$-ideal and $b\leqslant a$ with $a\in I$.
Then $a+b=a$, so $\M(a)=\M(a+b)=\M(a)\cap\M(b)$,
giving $\M(a)\subseteq\M(b)$.
With $d=ab\in I$, one has $\M(d)=\M(a)$ and hence $a\in I$,
and $\M(d)=\M(a)\cup\M(b)=\M(b)$, so $b\in I$.
\end{proof}

\subsection{The semiring \texorpdfstring{$\N$}{N}}

\begin{proposition}\label{app:N-z-ideal-expanded}
The unique maximal ideal of $\N$ is $\mathfrak{m}=\N\setminus\{1\}$,
and the only proper $z$-ideal of $\N$ is $\mathfrak{m}$.
\end{proposition}
\begin{proof}
The set $\mathfrak{m}$ is an ideal containing every non-unit,
hence maximal and unique.
A proper $z$-ideal containing any non-unit $a$ must contain every
non-unit (since $\M(a)=\{\mathfrak{m}\}=\M(b)$ for all $b\neq1$),
so it equals $\mathfrak{m}$.
The zero ideal is not a $z$-ideal: $\M(0)=\M(2)$ but $2\neq0$.
\end{proof}

\begin{proposition}\label{app:N-congruences-expanded}
For $d\geqslant2$, the congruence $k_d$ on $\N$ is a $z$-congruence
if and only if $d$ is square-free.  The Rees congruence
$\beta=(\N_{>0}\times\N_{>0})\cup\{(0,0)\}$ is maximal with
zero-class $\{0\}$, which is not a $z$-ideal.
\end{proposition}
\begin{proof}
Maximal congruences on $\N$ are classified as follows.
If $\theta$ is maximal proper and $0\sim_\theta d$ for the least
positive $d$, then $k_d\subseteq\theta$; maximality forces $d$ prime
and $\theta=k_p$.
If no positive integer is congruent to $0$, then $\theta\subseteq\beta$;
maximality gives $\theta=\beta$.

For $a,b\in\N$: $k_p\in\MC(a,b)\Leftrightarrow p\mid a-b$
and $\beta\in\MC(a,b)\Leftrightarrow a,b$ are both positive or
both zero.
The square-free criterion is Proposition~\ref{prop:N-zcong-squarefree}.
Finally, $0_\beta=\{0\}$, which is not a $z$-ideal by
Proposition~\ref{app:N-z-ideal-expanded}.
\end{proof}

\subsection{Finite power-set semirings}

\begin{proposition}\label{app:powerset-expanded}
Let $X$ be finite and $S=(\mathcal{P}(X),\cup,\cap,\varnothing,X)$.
For each $x\in X$, define
$(A,B)\in\Theta_x\Leftrightarrow[x\in A\text{ iff }x\in B]$.
The maximal congruences of $S$ are exactly $\{\Theta_x:x\in X\}$,
and every congruence on $S$ is a $z$-congruence.
\end{proposition}
\begin{proof}
Each $\Theta_x$ is maximal (its quotient has two elements).
Every congruence is determined by its zero-class (a downward-closed
ideal of the finite Boolean algebra), and maximality forces the
zero-class to be $\mathfrak{T}_x=\{A:x\notin A\}$,
hence $\rho=\Theta_x$.
If $(A,B)\in\rho$ and $\MC(A,B)\subseteq\MC(C,D)$,
then $A\triangle B\in I_\rho$ and $C\triangle D\subseteq A\triangle B$;
since $I_\rho$ is downward closed, $C\triangle D\in I_\rho$,
so $(C,D)\in\rho$.
\end{proof}

\subsection{Associated congruences and \texorpdfstring{$g$}{g}-closures}

\begin{lemma}\label{app:g-closure-operator}
The assignment $I\mapsto I^g:=0_{\wp_I}$ is an extensive,
order-preserving, idempotent closure operation on ideals.
\end{lemma}
\begin{proof}
This is Proposition~\ref{prop:g-closure-operator}.
\end{proof}

\begin{lemma}\label{app:MC-product-prime-zero}
Let $I,J$ be ideals of $S$.  Then $\MC(I+J)=\MC(I)\cap\MC(J)$.
If zero-classes of maximal congruences are prime ideals, then
\[
\MC(IJ)=\MC(I\cap J)=\MC(I)\cup\MC(J),
\]
and consequently $\wp_{I\cap J}=\wp_I\cap\wp_J=\wp_{IJ}$.
\end{lemma}
\begin{proof}
The equality for sums is immediate.
Let $\theta$ be a maximal congruence with $0_\theta$ prime.
If $I\subseteq0_\theta$ or $J\subseteq0_\theta$,
then both $IJ$ and $I\cap J$ are contained in $0_\theta$.
Conversely, if $IJ\subseteq0_\theta$ and neither ideal is contained
in $0_\theta$, choose $a\in I\setminus0_\theta$ and
$b\in J\setminus0_\theta$;
then $ab\in0_\theta$ contradicts primeness.
The same argument applies with $I\cap J$ in place of $IJ$.
Hence $\MC(IJ)=\MC(I\cap J)=\MC(I)\cup\MC(J)$,
and intersecting these families gives the formula for $\wp$.
\end{proof}

This lemma identifies the point at which primeness of zero-classes
enters the congruence theory.
In settings where maximal congruences are prime and prime congruences
have prime zero-classes, the prime-zero-class hypothesis is
automatic; in general semirings it is an additional condition.

\subsection{Quotients by associated congruences}

The quotient section is governed by representative-independence.
Without saturation, the expression $J/\wp_I$ may depend on the
chosen representatives and need not define an ideal in the quotient.

\begin{lemma}\label{app:quotient-saturation-expanded}
Let $J$ be an ideal of $S$.  The subset
$J/\wp_I=\{[a]\in S/\wp_I\mid a\in J\}$
is well-defined and an ideal of $S/\wp_I$ if and only if $J$
is $I$-saturated.
\end{lemma}
\begin{proof}
If the subset is a well-defined ideal:
$[0]\in J/\wp_I$ gives $I^g\subseteq J$,
and $(a,b)\in\wp_I$ with $a\in J$ gives $[b]=[a]\in J/\wp_I$,
so $b\in J$.
Conversely, under $I$-saturation:
if $[a]=[b]$ and $a\in J$, saturation gives $b\in J$,
so the subset is representative-independent;
$I^g\subseteq J$ ensures the zero class lies in $J$;
and closure under addition and scalar multiplication is inherited.
\end{proof}

\begin{proposition}\label{app:quotient-intersection-expanded}
Let $J,K,L$ be $I$-saturated ideals containing $I$.  Then
$(J/\wp_I)\cap(K/\wp_I)=L/\wp_I$ if and only if $J\cap K=L$.
\end{proposition}
\begin{proof}
If $[x]$ lies in both quotients, saturation of $J$ and $K$ gives
$x\in J\cap K$; if $x\in J\cap K$, then $[x]$ lies in both.
The correspondence with $L$ follows from the saturation bijection.
\end{proof}

\begin{proposition}\label{app:relative-z-quotient-expanded}
Assume $I$ is $g$-closed and satisfies
Hypothesis~\ref{hyp:quotient-lifting}.
Let $J$ be an $I$-saturated ideal containing $I$.
Then $J/\wp_I$ is a $z$-ideal of $S/\wp_I$ if and only if $J$ is
an $I$-relative $z$-ideal.
\end{proposition}
\begin{proof}
If $J/\wp_I$ is a $z$-ideal and $a\in J$ with
$\M(a)\cap\mathcal{V}_{\max}(I)=\M(b)\cap\mathcal{V}_{\max}(I)$,
Hypothesis~\ref{hyp:quotient-lifting} translates this into equal
maximal hulls of $[a]$ and $[b]$ in the quotient;
the $z$-ideal property gives $[b]\in J/\wp_I$,
and saturation gives $b\in J$.
The converse translates the quotient-side hull equality back to the
relative equality in $S$.
\end{proof}

\subsection{Localization}

The localization results require the maximal-disjointness hypothesis.
The standard ring-theoretic correspondence between maximal ideals of
$T^{-1}S$ and maximal ideals of $S$ disjoint from $T$ is used here
only under the explicit form stated in the body of the paper.

\begin{lemma}\label{app:local-hull-expanded}
Assume Hypothesis~\ref{hyp:local}.
For $a\in S$ and $s\in T$,
\[
\M_{T^{-1}S}(a/s)=\{T^{-1}M\mid M\in\mathcal{V}_{\max}(I)
\text{ and }a\in M\}.
\]
Consequently equality of $I$-relative maximal hulls in $S$ is
equivalent to equality of maximal hulls after localization.
\end{lemma}
\begin{proof}
Since $s/1$ is a unit, membership of $a/s$ in a maximal ideal equals
membership of $a/1$.
Hypothesis~\ref{hyp:local} identifies maximal ideals of $T^{-1}S$
with ideals $T^{-1}M$ for $M\in\mathcal{V}_{\max}(I)$,
and gives $a/s\in T^{-1}M\Leftrightarrow a\in M$.
\end{proof}

\begin{proposition}\label{app:local-z-expanded}
Assume Hypothesis~\ref{hyp:local}.
Let $J$ be an $I$-relative $z$-ideal.
Then $T^{-1}J$ is a $z$-ideal of $T^{-1}S$.
\end{proposition}
\begin{proof}
Let $x/s\in T^{-1}J$, so $ux\in J$ for some $u\in T$.
Suppose $\M_{T^{-1}S}(x/s)=\M_{T^{-1}S}(y/t)$.
If $M\in\M_S(ux)\cap\mathcal{V}_{\max}(I)$,
then $u\notin M$ by Hypothesis~\ref{hyp:local};
primeness of $M$ and $ux\in M$ give $x\in M$,
so $x/s\in T^{-1}M$, and the localized hull equality gives
$y/t\in T^{-1}M$, hence $y\in M$.
The reverse inclusion is analogous, giving
$\M_S(ux)\cap\mathcal{V}_{\max}(I)=\M_S(y)\cap\mathcal{V}_{\max}(I)$.
Since $J$ is $I$-relative and $ux\in J$, $y\in J$.
\end{proof}

\subsection{Coherence and spectra}

The frame theorem for ordinary $z$-ideals is a finite-type closure
theorem requiring no hypothesis beyond the semiring axioms.
The same conclusion for $g$-closed ideals requires
Hypothesis~\ref{hyp:g-closure}, and
Example~\ref{ex:zero-class-distinction} is one reason the
congruence-generated closure is treated separately from the ordinary
maximal-ideal closure.

\begin{lemma}\label{app:coherent-expanded}
In $\ZId(S)$, $\clz(I)=\bigvee_{a\in I}\ideal{a}_z$ and
$\ideal{a}_z\cap\ideal{b}_z=\ideal{ab}_z$ for all $a,b\in S$.
\end{lemma}
\begin{proof}
The join formula holds because a $z$-ideal contains $I$ if and only
if it contains every $\ideal{a}_z$ with $a\in I$.
For the meet:
$\ideal{a}_z=\mfrak(a)$ by Lemma~\ref{app:maximal-hull-calculus}(3),
and $\mfrak(ab)=\mfrak(a)\cap\mfrak(b)$ by the product formula.
\end{proof}

\begin{proposition}\label{app:compact-expanded}
The compact elements of $\ZId(S)$ are the finite joins
$\ideal{a_1}_z\vee\cdots\vee\ideal{a_n}_z$,
and they are closed under finite meets.
\end{proposition}
\begin{proof}
Each $\ideal{a}_z$ is compact by the finite-type property of the
maximal-hull closure.
Finite joins of compact elements are compact.
Conversely, if $K\in\ZId(S)$ is compact, then
$K=\bigvee_{a\in K}\ideal{a}_z$; compactness gives a finite subjoin.
Closure under finite meets follows from
$\ideal{a}_z\cap\ideal{b}_z=\ideal{ab}_z$ and the distributive law.
\end{proof}

The spectrum of a coherent frame is spectral.
Applying this to $\ZId(S)$ gives the spectrality of $\Spec_z(S)$.
Applying it to $\GCl(S)$ under Hypothesis~\ref{hyp:g-closure} gives
Theorem~C; this is why the latter result is stated conditionally.

\subsection{Functoriality}

The functorial formulas are determined on compact generators.
The following well-definedness lemma is included because
it is a common source of error in this setting.

\begin{lemma}\label{app:functor-expanded}
Let $\varphi:S\to T$ be a morphism in $\mathsf{CRig}_z$.
If $\ideal{a_1}_z\vee\cdots\vee\ideal{a_n}_z
=\ideal{b_1}_z\vee\cdots\vee\ideal{b_m}_z$ in $\ZId(S)$, then
$\ideal{\varphi(a_1)}_z\vee\cdots\vee\ideal{\varphi(a_n)}_z
=\ideal{\varphi(b_1)}_z\vee\cdots\vee\ideal{\varphi(b_m)}_z$
in $\ZId(T)$.
\end{lemma}
\begin{proof}
Let $L$ be the right-hand join in $\ZId(T)$.
Since $\varphi$ contracts $z$-ideals, $\varphi^{-1}(L)$ is a
$z$-ideal of $S$ containing each $b_j$,
hence containing the common join, hence containing each $a_i$.
Thus $\varphi(a_i)\in L$.
The reverse inclusion is symmetric.
\end{proof}

\begin{proposition}\label{app:naturality-expanded}
For a morphism $\varphi:S\to T$ in $\mathsf{CRig}_{g}^{\sigma}$,
the diagram
\[
\begin{tikzcd}
\ZId(S) \arrow[r,"\ZId(\varphi)"] \arrow[d,"\sigma_S"'] &
\ZId(T) \arrow[d,"\sigma_T"] \\
\GCl(S) \arrow[r,"\GCl(\varphi)"'] & \GCl(T)
\end{tikzcd}
\]
commutes.
\end{proposition}
\begin{proof}
Both composites preserve joins.
On a compact generator $\ideal{a}_z$,
the generator-compatibility condition for morphisms in
$\mathsf{CRig}_{g}^{\sigma}$ gives
\[
\GCl(\varphi)(\sigma_S(\ideal{a}_z))
=\GCl(\varphi)((\ideal{a}_z)^g)
=(\ideal{\varphi(a)}_z)^g
=\sigma_T(\ZId(\varphi)(\ideal{a}_z)).
\]
Since compact generators join-generate $\ZId(S)$,
the two composites agree everywhere.
\end{proof}

\section*{Acknowledgement}
The first author expresses sincere thanks to the Council of Scientific and Industrial Research (CSIR), Government of India, for the financial support (File No. 09/0096(17468)/2024-EMR-I). The third author is grateful to the University Grants Commission (India) for providing a Senior Research Fellowship (ID:
211610013222/ Joint CSIR-UGC NET JUNE 2021). The first, third, and fourth authors are also thankful to the DST-FIST Purse Programme (Programme no. SR/FST/MS-II/2021/101(C)) of the Department of Mathematics, Jadavpur University, Kolkata, India.

\end{document}